\documentclass[12pt]{amsproc}
\usepackage{amsmath}
\usepackage{amsthm,amssymb,amsfonts}
\usepackage[english]{babel}
\usepackage{paralist,url,verbatim, anysize}
\usepackage{amscd}
\usepackage[all,cmtip]{xy} 
\usepackage{mathtools}
\usepackage{leftidx}
\usepackage{graphicx}
\usepackage{euler}
\usepackage{csquotes}
\usepackage[bookmarks=false]{hyperref}
\usepackage{enumitem}
\usepackage{tikz-cd}

\theoremstyle{plain}
\newtheorem{theorem}{Theorem}[section]
\newtheorem{proposition}[theorem]{Proposition}
\newtheorem{notation}[theorem]{Notation}

\newtheorem{lemma}[theorem]{Lemma}
\newtheorem{remark}[theorem]{Remark}
\newtheorem{conjecture}[theorem]{Conjecture}

\newtheorem{definition}[theorem]{Definition}
\newtheorem{example}[theorem]{Example}
\newtheorem{question}[theorem]{Question}

\def\I{\mathcal{I}}
\def\J{\mathcal{J}}
\def\k{\Bbbk}
\def\K{\mathbb{K}}

\def\mM{\mathcal{M}}

\def\T{\mathcal{T}}

\DeclareMathOperator{\im}{im}

\DeclareMathOperator{\Spec}{Spec}
\DeclareMathOperator{\rank}{rank}

\DeclareMathOperator{\height}{height}

\makeatletter
    \def\@@and{ \ \ \ }

\let\sv@thm\@thm
\def\@thm{\let\indent\relax\sv@thm}

\renewenvironment{proof}[1][\proofname]{\par
  \pushQED{\qed}%
  \normalfont \topsep6\p@\@plus6\p@\relax
  \trivlist
  \itemindent\z@ 
  \item[\hskip\labelsep
        \scshape
    #1\@addpunct{.}]\ignorespaces
}{%
  \popQED\endtrivlist\@endpefalse
}
\def\subsection{\@startsection{subsection}{2}%
    \z@{.5\linespacing\@plus.7\linespacing}{-.5em}%
    {\normalfont\bfseries}}
\makeatother

\let\oldtimes\times
\renewcommand{\times}{\kern-1pt \oldtimes \kern-1pt}

\title{On The Determinantal Matroid}

\author[Lisa Nicklasson]{Lisa Nicklasson}
\address[Nicklasson]
{Division of Mathematics and Physics,
Mälardalen University}
\email{lisa.nicklasson@mdu.se}

\author[Manolis C. Tsakiris]{Manolis C. Tsakiris}
\address[Tsakiris]{Academy of Mathematics and Systems Science,
Chinese Academy of Sciences}
\email{manolis@amss.ac.cn}

\thanks{M.C. Tsakiris acknowledges the support of the National Key R$\&$D Program of China (2023YFA1009402) and of the CAS Project for Young Scientists in Basic Research (YSBR-
034).
}


\begin{document}

\begin{abstract}
We study the algebraic matroid induced by the ideal of $(r+1)$-minors of an $m \times n$ matrix of variables over a field. This is inherently connected to the bounded-rank matrix completion problem, in which the aim is to complete a partially observed $m \times n$ rank $r$ matrix. We give criteria that detect dependent sets in the matroid, we describe a family of bases of the matroid, and we study the question of unique completability.  
\end{abstract}

\maketitle

\section{Introduction} \label{section:Introduction}

The notion of a matroid, introduced independently by Whitney \cite{whitney1935abstract} and Nakasawa \cite{nakasawa1935axiomatik}-\cite{nakasawa1936axiomatik}, abstracts the idea of linear independence, and as such, is an effective combinatorial notion with a deep theory of its own and a multitude of applications within and outside mathematics. Algebraic matroids, already noted informally by van der Waerden \cite{van1937moderne}, appear naturally in field theory, with the transendence degree playing the role of the matroid rank. Algebraic matroids also naturally underpin the combinatorics of coordinate projections of affine varieties and the elimination theory of the corresponding polynomial ideals \cite{rosen2020algebraic}. In contrast to linear matroids, \emph{very little is known about algebraic matroids} \cite{dress1987some}, which \emph{seem forebodingly difficult to deal with} \cite{bollen2018algebraic}; this is essentially because understanding them is tantamount to understanding the algebraic relations that involve dependent elements, and while this is possible algorithmically for small examples, it is nevertheless a formidable theoretical task. In this paper we undertake that task for the \emph{determinantal matroid}, which is the algebraic matroid associated to the determinantal variety of $m \times n$ matrices over a field $\k$ of rank bounded above by $r < \min\{m,n\}$.


It is well-known \cite{kiraly2015algebraic,rosen2020algebraic} that the study of the determinantal matroid is of significance to the low-rank matrix completion problem, the latter being a fundamental question in the machine learning, statistics and applied mathematics communities; e.\,g.\ see \cite{candes2009exact,candes2010power,singer2010uniqueness,recht2011simpler,balcan2019non,tsakiris2023Plucker}. There, one observes an $m \times n$ matrix $X$ of rank $r<\min\{m,n\}$ over $\k$ at a subset $\Omega$ of its entries, and the objective is to complete these partial observations to a matrix of rank $r$. In other words, one applies a coordinate projection to an element of the determinantal variety and aims at obtaining an element of the fiber. There are finitely many rank-$r$ completions of $X$ if and only if the fiber of the projection over $X$ has dimension zero (thought over $\mathbb{C}$ or scheme-theoretically). The \emph{observation patterns} $\Omega$ of size equal to the dimension $r(m+n-r)$ of the determinantal variety, for which the generic fiber of the projection is zero-dimensional, are precisely the bases of the determinantal matroid; more generally, a set $\Omega$ is independent if and only if the associated coordinate projection is dominant. This connects directly with the feasibility question in matrix completion: Setting generic values at locations $\Omega$, can the rest of the matrix be completed to a matrix of rank $r$?\footnote{See also \cite{bernstein2020typical} for the related problem of completion ranks.} A more subtle question, yet of great interest for applications, is that of unique completability: When is $X$ the only element of the fiber over $X$?  

The determinantal matroid is theoretically understood only for $r=1$ \cite{singer2010uniqueness}, $r=2$ \cite{bernstein2017completion}, $r=\min\{m,n\}-3$ \cite{Larson2024} and $r=\min\{m,n\}-2, \, \min\{m,n\}-1$ \cite{Tsakiris-AMS-2023}. In \cite{Tsakiris-AMS-2023} two conditions were presented for independence, one is sufficient and the other is necessary; in \S \ref{section:Matroid-and-SLMF} we settle to the negative the open question of wether any of these is both necessary and sufficient. In \S\ref{section:General-height-criterion} we establish in great generality a criterion for dependence (Theorem \ref{thm:height}). In \S\ref{section:Detecting dependent sets} we further prove novel criteria for detecting dependent sets, with our main result being Theorem \ref{thm:dependent_Asche}. Moreover we conjecture that Theorem \ref{thm:height} detects all dependent sets (Conjecture \ref{conj:height}), and we prove this conjecture for $r=1$ and $r \ge \min\{m,n\}-2$ (Proposition \ref{prp:conjecture}). In \S \ref{section: Unique Completability} we lay down fundamental theory regarding the question of unique completability, and finally in \S \ref{section: Diagonal Bases} we discuss a family of maximal independent sets that we term \emph{diagonal bases}. 

We thank Aldo Conca for valuable discussions on the topic. We thank Chenqi Mou for stimulating interactions on Gr\"obner bases and matrix completion.

\section{Preliminaries and overview} \label{section:Background and notation}

We begin by setting our notation. Henceforth $Z$ will denote an $m \times n$ matrix of variables. For a positive integer $t$, we let $I_t(Z)$ denote the ideal of the polynomial ring $\k[Z]$ generated by all $t$-minors; here $\k$ is an infinite field and we denote its algebraic closure by $\K$. We recall that $I_t(Z)$ is a prime ideal and $$\height(I_{r+1}(Z)) = (m-r)(n-r).$$ A subset $\Omega \subseteq [m]\times [n]$ will often be depicted as a $\{0,1\}$-matrix, the \emph{incidence matrix} of $\Omega$, with a $1$ in entry $(i,j)$ indicating that $(i,j) \in \Omega$. For such an $\Omega$ we denote by $Z_\Omega=\{z_{ij} \ | \ (i, j) \in \Omega \}$ the corresponding subset of variables. We will use the short notation $E_\Omega$ for the elimination ideal $I_{r+1}(Z) \cap \k[Z_\Omega]$, and $(I_{r+1}(Z))_\Omega$ (or $I_\Omega$ more economically) for the ideal of $\k[Z_\Omega]$ generated by the $(r+1)$-minors supported on $\Omega$. Note that $I_\Omega \subseteq E_\Omega$, but the two ideals are not in general equal. By $\bar{Z}$ we denote the class of $Z$ modulo $I_{r+1}(Z)$, that is $\k[\bar{Z}] = \k[Z]/I_{r+1}(Z)$. We have a $\k$-algebra homomorphism $$\varphi_\Omega: \k[Z_\Omega] \rightarrow \k[\bar{Z}],$$ which for $(i,j) \in \Omega$ sends $z_{ij}$ to $\bar{z}_{ij}$. The kernel of $\varphi_\Omega$ is $E_\Omega$, and so we also have an inclusion of integral domains $\k[\bar{Z}_\Omega] \subseteq \k[\bar{Z}],$ where $\k[\bar{Z}_\Omega]  = \k[Z_\Omega]/ E_\Omega$, as well as a field extension $\k(\bar{Z}_\Omega) \subseteq \k(\bar{Z})$ of their fields of fractions. Now $\varphi_\Omega$ induces a projection morphism 
$$\pi_\Omega: \Spec \k[\bar{Z}] \rightarrow \Spec \k[Z_\Omega]=:\mathbb{A}^\Omega_\k.$$ 
Finally, by $(\Spec \k[\bar{Z}])_{\text{rat}}$ we indicate the set of $\k$-rational points of $\Spec \k[\bar{Z}]$; these are the $m \times n$ matrices with entries in $\k$ and of rank $\le r$. Note that $(\Spec \K[\bar{Z}])_{\text{rat}}$ coincides with the set of closed points of $\Spec \K[\bar{Z}]$. 

We recall that a \emph{matroid} is defined on a finite ground set $X$ by declaring each subset of $X$ either \emph{independent} or \emph{dependent}. Any subset of an independent set must also be independent. Moreover, if $A$ and $B$ are two independent sets such that $|A|<|B|$ then there is an $x \in B \setminus A$ such that $A \cup \{x\}$ is independent. As a consequence, all maximal independent sets must have the same cardinality, which is called the \emph{rank} of the matroid. The maximal independent sets are called \emph{bases}, as knowing the bases provides the full matroid structure. 
For fixed numbers $r, \, m, \, n$ we define a matroid $\mM(r,[m]\times [n])$ on ground set $[m] \times [n]$ by taking the subsets $\Omega \subseteq [m] \times [n]$ for which $E_\Omega = 0$ as independent. Equivalently, $\Omega$ is independent in $\mM(r,[m]\times [n])$ if $\pi_\Omega$ is dominant, or equivalently, if the elements $\bar{Z}_\Omega$ are algebraically independent over $\k$. For proofs that this is indeed a matroid, and of the said equivalences, see \cite{rosen2020algebraic}. The rank of $\mM(r,[m]\times [n])$ is 
$$\rank \mM(r,[m]\times [n]) = r(m+n-r) = \dim \Bbbk [Z]/I_{r+1}(Z) = \operatorname{tr.deg.}_\k \k(\bar{Z}).$$ 
An $\Omega$ of size $r(m+n-r)$ is a base if and only if the field extension $\k(\bar{Z}_\Omega) \subseteq \k(\bar{Z})$ is finite, or equivalently, if and only if the generic fiber of $\pi_\Omega$ is zero-dimensional. 
One may also define the  \emph{matroid rank} of any $\Omega \subseteq [m] \times [n]$ as $$\rank(\Omega) = \operatorname{tr.deg.}_\k \k(\bar{Z}_\Omega) = \dim (\im(\pi_\Omega)).$$ 

\begin{definition} \label{dfn:finitely-completable}
An $\Omega \subseteq [m] \times [n]$ is called finitely completable at rank $r$, if the generic fiber of $\pi_\Omega$ is zero-dimensional. 
\end{definition} 

\begin{remark}
If $\Omega$ is finitely completable, the minimal fiber dimension of $\pi_\Omega$ is zero; by the upper semicontinuity of the fiber dimension, this minimal value is achieved on a dense open set, and from this it follows that a generic $m \times n$ rank $r$ matrix with entries in $\k$ observed at $\Omega$ has finitely many rank $r$ completions over $\K$. In such a case, the matroid rank of $\Omega$ is the maximal possible and so $\Omega$ contains a base of $\mathcal{M}(r, [m] \times [n])$. Thus understanding the finitely completable $\Omega$'s is equivalent to understanding the bases of $\mathcal{M}(r, [m] \times [n])$.
\end{remark}

Our study of $\mathcal{M}(r, [m] \times [n])$ begins in \S \ref{section:Matroid-and-SLMF}, where we give examples proving that, in general, none of the conditions given in \cite{Tsakiris-AMS-2023} are necessary and sufficient for an $\Omega$ to be a base. Motivated by this gap, we develop in  \S \ref{section:General-height-criterion} a general criterion for dependence, valid for any algebraic matroid induced by a prime ideal in a polynomial ring. In \S \ref{section:Detecting dependent sets} we establish further criteria for detecting dependent sets in $\mathcal{M}(r, [m] \times [n])$, while in \S \ref{section: Diagonal Bases} we present a special family of \emph{diagonal} and \emph{anti-diagonal} bases of $\mathcal{M}(r, [m] \times [n])$. In \S \ref{section: Unique Completability} we are concerned with the following more subtle notion.

\begin{definition} \label{dfn:uniquely-completable}
An $\Omega \subseteq [m] \times [n]$ is called uniquely completable at rank $r$, if the generic fiber of $\pi_\Omega$ consists of a single point (scheme-theoretically). 
\end{definition} 

Among others, we will see in \S \ref{section: Unique Completability} that the scheme-theoretic notion of unique completability in Definition \ref{dfn:uniquely-completable} is in fact equivalent to the one expected from a matrix completion point of view, which may be formulated as: 

\begin{definition} \label{dfn:generically-uniquely-completable-k}
$\Omega$ is called generically uniquely completable at rank $r$ over $\k$ if there exist polynomials in $\k[Z] \setminus I_{r+1}(Z)$, such that if $X=(x_{ij}) $ is an $m \times n$ matrix with entries in $\k$ of rank at most $r$ and $X$ is not a simultaneous root of these polynomials, then $X = Y$ for any other matrix $Y=(y_{ij}) \in \K^{m \times n}$ that satisfies $y_{ij} = x_{ij}, \, \forall (i,j) \in \Omega$ and $\rank(Y) \le r$.
\end{definition}

\section{The determinantal matroid and supports of linkage matching fields} \label{section:Matroid-and-SLMF}

\emph{Supports of Linkage Matching Fields} (SLMF's for short) is an algebraic-combinatorial notion introduced in \cite{sturmfels1993maximal} along the effort of proving that the maximal minors of a matrix of variables is a universal Gr\"obner basis. It was generalized in \cite{Tsakiris-AMS-2023} to the notion of \emph{relaxed-SLMF} and used in the study of the matroid $\mathcal{M}(r, [m] \times [n])$.

\begin{definition}[\cite{Tsakiris-AMS-2023}] \label{definition:SLMF}
A subset $\Omega \subseteq [m] \times [n]$ is a \emph{relaxed $(\nu,r,m)$-SLMF} if for all $\I \subseteq [m]$, with $|\I|>r$, 
\begin{equation}\label{eq:rrmSLMF}
\sum_{j=1}^n \max(|\Omega \cap (\I \times \{j\})|-r,0) \le \nu(|\I|-r) \quad \text{with equality for} \ \I=[m].
\end{equation}
\end{definition}

The following sufficient condition was proved for $\Omega$ to be a base.

\begin{theorem}[\cite{Tsakiris-AMS-2023}] \label{theorem:SLMF-sufficient}
If $|\Omega| = r(m+n-r)$ and there is a partition $[n] = \bigcup_{\ell= 1}^r \J_\ell$ such that each $\Omega \cap ([m]\times \J_\ell)$ is a relaxed $(1,r,m)$-SLMF, then $\Omega$ is a base of the matroid $\mathcal{M}(r, [m] \times [n])$. 
\end{theorem}

A necessary condition was also given.

\begin{theorem} [\cite{Tsakiris-AMS-2023}] \label{theorem:SLMF-necessary}
Every base of $\mathcal{M}(r, [m] \times [n])$ is a relaxed $(r,r,m)$-SLMF.
\end{theorem}

By \cite[Lemma 7]{Tsakiris-AMS-2023} it is sufficient to consider bases for which each row and column in the incidence matrix has at least $r+1$ ones, see also Lemma \ref{lemma:indep_reduce}.
Under this harmless assumption, the two conditions were proved to be equivalent in the extreme cases $r=1, \, \min\{m,n\}-2, \, \min\{m,n\}-1$, with the question of overall equivalence left open. Here we give examples which show that, under the same assumption, neither condition characterizes the matroid.  

\begin{example}\label{ex:not-1rm-SLMF}
 Let $r=4$ and consider $\Omega$ given by the $9 \times 9$ incidence matrix 
 \[
 \begin{bmatrix}
     0 & 0 & 0 & 1 & 1 & 1 & 1 & 1 & 1 \\
     0 & 0 & 1 & 1 & 0 & 1 & 1 & 1 & 1 \\
     0 & 1 & 1 & 0 & 1 & 1 & 0 & 1 & 1 \\
     1 & 1 & 0 & 0 & 1 & 0 & 1 & 1 & 1 \\
     1 & 0 & 1 & 1 & 1 & 0 & 1 & 0 & 1 \\
     1 & 1 & 1 & 0 & 0 & 1 & 1 & 0 & 1 \\
     1 & 1 & 0 & 1 & 1 & 1 & 0 & 0 & 1 \\
     1 & 1 & 1 & 1 & 0 & 0 & 0 & 1 & 1 \\
     1 & 1 & 1 & 1 & 1 & 1 & 1 & 1 & 0
 \end{bmatrix}.
\]
This is an example of an ``anti-diagonal'' base, which we will see in Theorem \ref{thm:anti-diagonal}. Now we shall observe that this set $\Omega$ does not satisfy the condition in Theorem \ref{theorem:SLMF-sufficient}.  Take a partition $[9]=\J_1 \cup \J_2 \cup \J_3 \cup \J_4$, and consider $\I=[9]$ in \eqref{eq:rrmSLMF} for $\Omega \cap ([9] \times \J_1)$. As every column has either six or eight elements, the terms in left hand side of \eqref{eq:rrmSLMF} can be either 2 or 4, so the sum is an even number. The right hand side is $9-4=5$, an odd number. Hence the equality can not hold, and we can not partition $\Omega$ into $(1,4,9)$-SLMF's in the way Theorem \ref{theorem:SLMF-sufficient} requires. 
\end{example}

\begin{example} \label{ex:SLMF-but-not-base}
Let $m=n=5$ and $r=2$, and 
consider 
\[
\Omega = \begin{pmatrix}
1 & 1 & 1 & 0 & 0   \\
1 & 1 & 1 & 0 & 0   \\
1 & 1 & 0 & 1 & 1   \\
0 & 0 & 1 & 1 & 1 \\
0 & 0 & 1 & 1 & 1
\end{pmatrix}.
\] Then a simple exhaustive by hand calculation shows that $\Omega$ is a relaxed $(2,2,5)$-SLMF, i.e. it satisfies the condition of Theorem \ref{theorem:SLMF-necessary}; in particular $|\Omega| = 16 = \rank \mathcal{M}(2,[5] \times [5])$. 
On the other hand, $\Omega$ is a dependent set of the matroid $\mathcal{M}(2,[5] \times [5])$. To see this, let $f$ and $g$ be the 3-minors $[123|123]$ and $[345|345]$. Notice that 
\[ 
f=f_1z_{33}+f_2 \quad \text{and} \quad g=g_1z_{33}+g_2 \quad \text{where} \quad f_1, f_2, g_1, g_2 \in \k[\Omega]. 
\]
Then 
\[
g_1f-f_1g= g_1f_2-f_1g_2
\]
is a non-zero element of the elimination ideal $E_{\Omega}$.
\end{example}

\section{A general height criterion for dependence}\label{section:General-height-criterion}

Quite generally, if $Z=\{z_1, \ldots, z_N\}$ is a set of $N$ variables, any prime ideal $I$ of the polynomial ring $\k[Z]$ defines an algebraic matroid as in the previous section, i.e. by declaring $\T \subseteq [N]$ to be independent if $E_\T = I \cap \k[Z_\T] = 0$. 
If the ideal $I$ is not prime, then the above matroid structure may not be present; instead, the sets $\T$ for which $E_\T=0$ define a simplicial complex on $[N]$ known as the \emph{independence complex} of $I$; e.g. see \cite{kalkbrener1995initial}. A fascinating result, proved  independently by different methods in \cite{varbaro2011symbolic} and \cite{minh-trung-2011}, asserts that the independence complex of a square-free monomial ideal is a matroid if and only if all symbolic powers of the ideal are Cohen-Macaulay. In this level of generality, we have the following \emph{dependence} criterion. 

\begin{theorem} \label{thm:height}
Let $I=(p_1, \ldots, p_s)$ be a proper ideal of $\k[Z]$, and for any $\T \subseteq [N]$ let $I_\T$ denote the ideal generated by the subset $\{p_1, \ldots, p_s\} \cap \k[Z_\T].$
An $\Omega \subseteq [N]$ is not in the independence complex of $I$ if there is a set $\T \subseteq [N]$ for which 
\begin{equation*}
\height(I_\T) > |\T \setminus \Omega|.
\end{equation*}   
\end{theorem}

\begin{proof} Given $\Omega$ suppose there is a $\T$ for which
$\height(I_\T) > |\T \setminus \Omega|$.  Let $S$ be the multiplicatively closed subset $\k[Z_{\T \cap \Omega}]\setminus \{0\}$ of the polynomial ring $\k[Z_{\T}]$. Localizing $\k[Z_{\T}]=\k[Z_{\T \cap \Omega}][Z_{\T \setminus \Omega}]$ at $S$, we obtain the ring $\k[Z_{\T}]_S = \k(Z_{\T \cap \Omega})[Z_{\T \setminus \Omega}]$, which is a polynomial ring of dimension $|\T \setminus \Omega|$ over the field $\k(Z_{\T \cap \Omega})$ of rational functions. 

Quite generally, if $J$ is a proper ideal of a commutative ring $R$ and $T$ a multiplicatively closed set of $R$, then $\height_{R_T} (J R_T) \ge \height_R(J)$, provided that $J R_T$ is a proper ideal of $R_T$. 
Consequently, if $I_\T \k[Z_{\T}]_S$ is a proper ideal of $\k[Z_{\T}]_S$, the hypothesis on $\T$ gives 
$$ \height_{\k[Z_{\T}]_S} (I_\T \k[Z_{\T}]_S) \ge \height_{\k[Z_{\T}]}(I_\T) >  |\T \setminus \Omega| = \dim \k[Z_{\T}]_S,$$ 
which is a contradiction. Therefore it must be that $I_\T \k[Z_{\T}]_S = \k[Z_{\T}]_S$. This implies that the identity element of $\k[Z_{\T}]_S$ can be written as a quotient $f/g$, with $f \in I_\T \k[Z_{\T}]$ and $g \in \k[Z_{\T \cap \Omega}] \setminus \{0\}$. Then $f$ is a non-zero polynomial in $I_\T \cap \k[Z_{\T \cap \Omega}] \subseteq I \cap \k[Z_{\Omega}]$, and so $E_\Omega \neq 0$; i.e. $\Omega$ is not in the independence complex of $I$.
\end{proof}

\begin{remark}
The matroid defined by a prime ideal $I$ is invariant of the chosen generating set for $I$, but the dependent sets detected by Theorem \ref{thm:height} do depend on the choice of generators. For example, take $I=(z_1z_2+z_3, \, z_3-z_4)$ and $\Omega=\{z_3,\, z_4\}$. Here $\Omega$ is a dependent set of the matroid, as $z_3-z_4 \in I$, and $\height(I_\T)=1 > 0=|\T\setminus \Omega|$ for $T=\Omega$.  On the other hand, we can take the generating set $\{z_1z_2+z_3, \, z_1z_2+z_4\}$ for the same ideal $I$. With these generators $\height(I_\T) \le |\T\setminus \Omega|$ for any $\T \subseteq \{z_1, z_2, z_3, z_4\}$, so Theorem \ref{thm:height} does not detect $\Omega$ as a dependent set. This also shows that the converse of Theorem \ref{thm:height} does not hold in general. 
\end{remark}

\section{Detecting dependent sets} \label{section:Detecting dependent sets}

We now return to the study of the \emph{Determinantal Matroid} introduced in \S \ref{section:Background and notation}. In particular, in this section we discuss ways of detecting dependent sets in the determinantal matroid. As a first result, Theorem \ref{thm:height} specialized to the determinantal ideal yields the following.

\begin{theorem} \label{thm:height-determinantal}
An $\Omega \subseteq [m] \times [n]$ is dependent in the matroid $\mM(r,[m]\times [n])$, if there is a set $\T \subseteq [m] \times [n]$ for which 
\begin{equation*}
\height(I_\T) > |\T \setminus \Omega|.
\end{equation*}   
\end{theorem}

\begin{example}
Consider again the $\Omega$ of Example \ref{ex:SLMF-but-not-base}.
We can apply Theorem \ref{thm:height-determinantal} with $\T=(\{1,2,3\} \times \{1,2,3\})\cup (\{3,4,5\} \times \{3,4,5\})$ to see that $\Omega$ is dependent. The ideal  $I_\T$ is generated by two minors, and as minors are irreducible polynomials, the two minors form a regular sequence. Hence $\height(I_\T)=2$ while $|\T \setminus \Omega|=1$. 
\end{example}

The following result from matroid theory provides a tool for detecting new dependent sets, given some known dependent sets of the matroid. 

\begin{theorem}[{\cite[Theorem 3]{Rado}, \cite[Theorem 4]{Asche}}] \label{thm:Asche}
Let $A_1, \ldots, A_s$ be dependent sets of a matroid on ground set $X$, such that $A_i \cap (\bigcup_{j<i}A_j)$ are independent for $i=2, \ldots, s$. Then for any subset $B \subseteq X$ with $|B|<s$ the set $\big(\bigcup_{i=1}^s A_i\big) \setminus B$ is dependent. 
\end{theorem}

Lemma \ref{lemma:basic_indep}-\ref{lemma:size_dep} below provide us with easily recognized independent and dependent sets. We will then use these as building blocks for constructing new dependent sets applying Theorem \ref{thm:Asche}. The idea is illustrated in 
Example \ref{ex:rank2_dep} below.

\begin{lemma}\label{lemma:indep_reduce}
    Let $\Omega \subset [m] \times [n]$ and consider 
    \[
    \J = \{ j \in [n] \ | \ |\Omega \cap ( [m] \times \{j\})|>r \}.
    \]
    Then $\Omega \cap ( [m] \times \J)$ is independent if and only if $\Omega$ is.
\end{lemma}
\begin{proof}
Let $\Omega_1=\Omega \cap ( [m] \times \J)$. It is clear that $\Omega_1$ is independent if $\Omega$ is. Next, assume $\Omega_1$ is independent. Then $\Omega_1$ can be extended to a base $\Omega_2$ of the matroid $\mM(r,[m] \times \J)$. Next choose a set $\Omega_3$ such that 
\begin{itemize}[itemsep=5pt, topsep=5pt]
 \item $\Omega_3 \supseteq \Omega $,
    \item $\Omega_3 \cap ( [m] \times \J) = \Omega_2$, and
    \item $|\Omega_3 \cap ( [m] \times \{j\})|=r$ for $j \notin \J$.
\end{itemize}
By \cite[Lemma 7]{Tsakiris-AMS-2023} the set $\Omega_3$ is a base if and only if $\Omega_3 \cap ( [m] \times \J)$ is a base of the matroid $\mM(r,[m] \times \J)$. Since $\Omega_3 \cap ( [m] \times \J)=\Omega_2$ is indeed a base, we get that $\Omega_3$ is a base. As $\Omega_3 \supseteq \Omega $ we can conclude that $\Omega$ is independent. 
\end{proof}

\begin{lemma}\label{lemma:basic_indep}
Consider subsets $\I \subseteq [m]$, $\J \subseteq [n]$, and $\Omega \subseteq \I \times \J$. If $|\I|\le r$ or $|\J|\le r$ then $\Omega$ is an independent set of the matroid $\mM(r,[m] \times [n])$.   
\end{lemma}
\begin{proof}
    By Lemma \ref{lemma:indep_reduce} $\Omega$ is independent if it is an independent set of the matroid $\mM(r,\I \times \J)$. But as $|\I|\le r$ or $|\J|\le r$ the ideal of minors $I_{r+1}(Z_{\I \times \J})$ is the zero ideal, which makes the elimination ideal $E_{\Omega}$ zero as well.   
\end{proof}

\begin{lemma}\label{lemma:small_block}
Let $\Omega \subseteq [m] \times [n]$ and suppose $\Omega \subseteq \I \times \J$ such that $|\I|=|\J|=r+1$. Then $\Omega$ is dependent in the matroid $\mM(r,[m]\times [n])$ if and only if $\Omega = \I \times \J$.    
\end{lemma}
\begin{proof}
By Lemma \ref{lemma:basic_indep} $\Omega$ is dependent if and only if it is a dependent set of the matroid $\mM(r,\I \times \J)$. It is proved in \cite{Tsakiris-AMS-2023}, the only dependent set if this matroid is $\I \times \J$. 
\end{proof}

\begin{lemma}\label{lemma:size_dep}
Let $\Omega \subseteq [m] \times [n]$ and consider subsets $\I \subseteq [m]$, $\J \subseteq [n]$ such that $|\I|>r$ and $|\J|>r$. If 
\begin{equation}\label{eq:lemma_size}
|\Omega \cap (\I \times \J)| >r(|\I|+|\J|-r), \ \text{or equivalently} \ |(\I \times \J) \setminus \Omega|< (|\I|-r)(|\J|-r),
\end{equation}
 then $\Omega$ is a dependent set.  
\end{lemma}
\begin{proof}
The equivalence of the two inequalities \eqref{eq:lemma_size}
follows from the fact that 
\[|\I \times \J| - r(|\I|+|\J|-r) = (|\I|-r)(|\J|-r).\]
Let $\Omega' = \Omega \cap (\I \times \J)$. 
The bases of the smaller matroid $\mM(r,\I \times \J)$ have size $r(|\I|+|\J|-r)$, so if $|\Omega'| >r(|\I|+|\J|-r)$ then $\Omega'$ is a dependent set of this matroid. Then $I_{r+1}(Z_{\I \times \J}) \cap k[Z_{\Omega'}] \ne 0$, which makes $\Omega'$ a dependent set of the matroid $\mM(r,[m] \times [n])$ as well. As $\Omega' \subseteq \Omega$, also $\Omega$ is a dependent set.
\end{proof}

We now revisit Example \ref{ex:SLMF-but-not-base} and show once again that $\Omega$ is dependent, this time using Theorem \ref{thm:Asche}. 

\begin{example}\label{ex:rank2_dep}
Let $m$, $n$, $r$, and $\Omega$ be as in Example \ref{ex:SLMF-but-not-base}. 
In Theorem \ref{thm:Asche} take $A_1=\{1,2,3\} \times \{1,2,3\}$ and $A_2=\{3,4,5\} \times \{3,4,5\}$, 
and $B=\{(3,3)\}$. 
By Lemma \ref{lemma:small_block} the sets $A_1$ and $A_2$ are dependent, while $A_1 \cap A_2 $ is independent. 
It follows that $\Omega=(A_1 \cup A_2) \setminus B$ is dependent. 
\end{example}

Generalizing the idea from Example \ref{ex:rank2_dep} we are able to prove the following.

\begin{theorem}\label{thm:dependent_Asche} 
Let $\Omega \subseteq [m] \times [n]$ and assume there are sets $\I_\alpha \times \J_\alpha \subseteq [m] \times [n]$ with $|\I_\alpha|,|\J_\alpha|>r$ for $\alpha=1, \ldots, s$
 such that 
\begin{equation}\label{eq:condition_dep_Asche}
(\I_\alpha \times \J_\alpha) \cap \biggl( \bigcup_{k < \alpha} \I_k \times \J_k \biggl) \subseteq \I'_\alpha \times \J'_\alpha \quad  \text{where} \ \min(|\I'_\alpha|, |\J'_\alpha|) \le r \
\end{equation}
 for each  $1<\alpha \le s$, and
\begin{equation}\label{eq:ineq_dep_Asche}
\left| \bigcup_{ \alpha= 1}^s \I_\alpha \times \J_\alpha  \setminus \Omega \right| < \sum_{\alpha=1}^s(|\I_\alpha|-r)(|\J_\alpha|-r). 
\end{equation}
Then $\Omega$ is a dependent set of the matroid $\mM(r,[m] \times [n])$. 
\end{theorem}
\begin{proof}

If $s=1$ then the result follows directly from Lemma \ref{lemma:size_dep}. Assume $s>1$. 

The first part of the proof is to replace each $\I_\alpha \times \J_\alpha$ by $(|\I_\alpha|-r)(|\J_\alpha|-r)$ submatrices
\[
\I^{ij}_\alpha \times \J^{ij}_\alpha \qquad  1 \le i \le   |\I_\alpha|-r , \  1 \le j \le   |\J_\alpha|-r
\] 
of type $(r+1) \times (r+1)$ such that 
\begin{equation}\label{eq:proofAsche1}
\bigcup_{i,j} \I^{ij}_\alpha \times \J^{ij}_\alpha = \I_\alpha \times \J_\alpha \quad \text{and}
\end{equation}
\begin{equation}\label{eq:proofAsche2}
(\I^{ij}_\alpha \times \J^{ij}_\alpha) \cap \left(  \left( \bigcup_{(k,\ell) \ne (i,j)} \I^{k \ell}_\alpha \times \J^{k \ell }_\alpha \right) \cup \left( \bigcup_{\beta<\alpha} \I_\beta \times \J_\beta \right) \right) \quad \text{is independent.}
\end{equation}
To avoid a heavy notation, we may assume 
\[
\I_\alpha=\{ 1,2, \ldots, |\I_\alpha|\}, \quad \J_\alpha=\{ 1,2, \ldots, |\J_\alpha|\}
\]
and if $\alpha>1$ 
\[
\I'_\alpha=\{ 1,2, \ldots, |\I'_\alpha|\}, \quad \J'_\alpha=\{ 1,2, \ldots, |\J'_\alpha|\},
\]
where we recall that $ \min \big(|\I'_\alpha|, |\J'_\alpha|\big) \le r$.
Let
\[
 \I^{ij}_\alpha \times \J^{ij}_\alpha = \{1,2, \ldots, r, r+i\} \times \{1, 2, \ldots, r, r+j\} \quad \text{for} \ 1 \le i \le   |\I_\alpha|-r , \  1 \le j \le   |\J_\alpha|-r.
\]
Notice that $(r+i, \, r+j) \in  \I^{k \ell}_\alpha \times \J^{k \ell }_\alpha$ if and only if $(k, \ell) = (i, j)$. Moreover, the element $(r+i,\, r+j)$ is not in  $\I'_\alpha \times \J'_\alpha$ and hence not in $\I_\beta \times \J_\beta$ for any $\beta<\alpha$. It follows that the set \eqref{eq:proofAsche2} is a subset of $\I^{ij}_\alpha \times \J^{ij}_\alpha$ not containing $(r+i, \, r+j)$, so by Lemma \ref{lemma:small_block} the set   \eqref{eq:proofAsche2} is independent. The equality \eqref{eq:proofAsche1} is clear from the definition of $ \I^{ij}_\alpha \times \J^{ij}_\alpha$. 

The second part of the proof is to apply Theorem \ref{thm:Asche}. Let $S= \sum_{\alpha=1}^s(|\I_\alpha|-r)(|\J_\alpha|-r)$. We choose $A_1, \ldots, A_S$ so that the first $(|\I_1|-r)(|\J_1|-r)$ of the $A_k$'s are the $\I^{ij}_1 \times \J^{ij}_1 $'s. The next $(|\I_2|-r)(|\J_2|-r)$ of the $A_k$'s are the $\I^{ij}_2 \times \J^{ij}_2$'s, and so on. By Lemma \ref{lemma:small_block} each $A_i$ is dependent, and \eqref{eq:proofAsche2} together with Lemma \ref{lemma:basic_indep} implies that each set $A_i \cap ( \bigcup_{j<i}A_j)$ is independent. By \eqref{eq:proofAsche1} we have $\bigcup_{i=1}^SA_i = \bigcup_{\alpha=1}^s\I_\alpha \times \J_\alpha$. Now choose $B=\bigcup_{i=1}^SA_i \setminus \Omega$. Then by \eqref{eq:ineq_dep_Asche} we have $|B|<S$ and we can apply Theorem \ref{thm:Asche} to conclude that 
$\bigcup_{i=1}^SA_i \setminus B = \Omega$
is a dependent set. 
\end{proof}

\begin{remark}\label{rmk:theorems_detect_basic_dep}
    The dependent sets described in Lemma \ref{lemma:size_dep} are detected both by Theorem \ref{thm:height-determinantal} and Theorem \ref{thm:dependent_Asche}. Indeed, we saw in the proof of Theorem \ref{thm:dependent_Asche} that Lemma \ref{lemma:size_dep} corresponds to the case $s=1$. As for Theorem \ref{thm:height-determinantal}, taking $\T=\I \times \J$ and recalling that $\height(I_\T)=(|\I|-r)(|\J|-r)$, one of the equivalent inequalities  in Lemma \ref{lemma:size_dep} states precisely that $|\T \setminus \Omega|<\height(I_\T)$.
\end{remark}

In Example \ref{ex:rank4_dep_ladder} we see another dependent set detected by both Theorem \ref{thm:height-determinantal} and Theorem \ref{thm:dependent_Asche}. Here the application of Theorem \ref{thm:Asche} is not as immediate as in Example \ref{ex:rank2_dep}.

\begin{example}\label{ex:rank4_dep_ladder}
Consider $\Omega$ given by 
\[
\begin{pmatrix}
1 & 1 & 1 & 1 & 1 & \multicolumn{1}{c|}{1} & 0 & 0 & 0 \\
1 & 1 & 1 & 1 & 1 & \multicolumn{1}{c|}{1} & 0 & 0 & 0 \\
1 & 1 & 1 & 1 & 1 & \multicolumn{1}{c|}{1} & 0 & 0 & 0 \\
1 & 1 & 1 & 1 & 1 & \multicolumn{1}{c|}{1} & 0 & 0 & 0 \\
\cline{7-9}
1 & 1 & 1 & 1 & 0 & 0 & 1 & 1 & 1 \\
1 & 1 & 1 & 1 & 0 & 0 & 1 & 1 & 1 \\
\cline{1-4}
0 & 0 & 0 & 0 & \multicolumn{1}{|c}{1} & 1 & 1 & 1 & 1 \\
0 & 0 & 0 & 0 & \multicolumn{1}{|c}{1} & 1 & 1 & 1 & 1 \\
0 & 0 & 0 & 0 & \multicolumn{1}{|c}{1} & 1 & 1 & 1 & 1 
\end{pmatrix}
\]
where $r=4$. Let $\T=(\I_1 \times \J_1) \cup (\I_2 \times \J_2)$ with 
\[
\I_1=\J_1=\{   1, \ldots,6  \} \quad \text{and} \quad \I_2 = \J_2=\{  5, \ldots, 9 \},
\]
and note that $|\T \setminus \Omega|=4$.
To apply Theorem \ref{thm:height-determinantal} we need to compute the height of the ideal $I_\T$. The set $\T$ can be considered a ladder (see \S \ref{section: Diagonal Bases} for a definition), and we can use \cite[Theorem 1.15]{Gorla} to compute
\[
\height(I_\T) = | \{1, 2\} \times \{5, 6\} \cup \{(5,9)\}| = 5. 
\]
It then follows that $\Omega$ is dependent as $5>|\T \setminus \Omega|$. 

Alternatively, we can compute 
\[
(|\I_1|-r)(|\J_1|-r)+(|\I_2|-r)(|\J_2|-r)=4+1=5
\]
and apply Theorem \ref{thm:dependent_Asche} to see that $\Omega$ is dependent. 
\end{example}

We now have both Theorem \ref{thm:height} and Theorem \ref{thm:dependent_Asche} for detecting dependent sets in the matroid $\mM(r, [m]\times [n])$. 
Taking $\T=\bigcup_{\alpha=1}^s\I_\alpha \times \J_\alpha$, both theorems state that $\Omega$ is dependent if $|\T \setminus \Omega|$ is smaller than a certain number. 
As 
\[
\height( I_{\I \times \J}) = (|\I|-r)(|\J|-r)
\]
one might wonder if the two numbers to which $|\T \setminus \Omega|$ is compared are the same. This is indeed the case in all examples in this section, but we can see in Example \ref{ex:height_formula} that it is not true in general.

\begin{example}\label{ex:height_formula}
Let $r=2$, and $\I_1=\J_1=\I_2=\{1,2,3\}$ and $\J_2=\{3 , 4, 5\}$. Then 
$$\T=(\I_1 \times \J_1) \cup (\I_2 \times \J_2) = \{1,2,3\} \times \{1,2,3,4,5\}$$ 
and  
 $\height(I_\T)=3$ while $(|\I_1|-r)(|\J_1|-r)+(|\I_2|-r)(|\J_2|-r)=2$.
\end{example}
This raises the following question.

\begin{question}\label{que:height_ineq}
     Suppose $\I_1 \times \J_1, \ldots, \I_s \times \J_s$ satisfies condition \eqref{eq:condition_dep_Asche} of Theorem \ref{thm:dependent_Asche}, and let $\T= \bigcup_{\alpha=1}^s \I_\alpha \times \J_\alpha$. Is it true that
\[ \height(I_\T) \ge \sum_{\alpha=1}^s \height(I_{\I_\alpha \times \J_\alpha}) \ ? \]%
\end{question}%
If the inequality in Question \ref{que:height_ineq} would hold then Theorem \ref{thm:height} would provide an alternative proof of Theorem \ref{thm:dependent_Asche}. 

In Example \ref{ex:rank5_dep}  the set $\T=\bigcup \I_\alpha \times \J_\alpha$ is not a ladder, and there is no direct formula for computing $\height(I_\T)$. 

\begin{example}\label{ex:rank5_dep}
Let $m=n=12$ and $r=5$, and consider the set $\Omega$ given by
\setcounter{MaxMatrixCols}{20}
\[
\begin{pmatrix}
1  & 1 & 1  & \multicolumn{1}{c|}{1} & 1 & \multicolumn{1}{c|}{1} & 1 & \multicolumn{1}{c|}{1} & 1 & \multicolumn{1}{c|}{1} & 1 & 1 \\
1  & 0 & 1  & \multicolumn{1}{c|}{1} & 1 & \multicolumn{1}{c|}{1} & 1 & \multicolumn{1}{c|}{1} & 1 & \multicolumn{1}{c|}{1} & 1 & 1 \\
1  & 1 & 0  & \multicolumn{1}{c|}{1} & 1 & \multicolumn{1}{c|}{1} & 1 & \multicolumn{1}{c|}{1} & 1 & \multicolumn{1}{c|}{1} & 1 & 1 \\
1  & 1 & 1  & \multicolumn{1}{c|}{0} & 1 & \multicolumn{1}{c|}{1} & 1 & \multicolumn{1}{c|}{1} & 1 & \multicolumn{1}{c|}{1} & 1 & 1 \\
\cline{1-4} \cline{7-12}
1 & 1 & 1 & 1 & 1 & 1 & \multicolumn{1}{|c}{0} & 0 & 0 & 0 & 0 & 0 \\
1  & 1 & 1 & 1 & 1 & 1 & \multicolumn{1}{|c}{0} & 0 & 0 & 0 & 0 & 0 \\
\cline{1-8}
1  & 1 & 1 & 1 &  \multicolumn{1}{|c}{0} & 0 &  \multicolumn{1}{|c}{1} & 1 &  \multicolumn{1}{|c}{0} & 0 & 0 & 0 \\
1  & 1 & 1 & 1 &  \multicolumn{1}{|c}{0} & 0 &  \multicolumn{1}{|c}{1} & 1 &  \multicolumn{1}{|c}{0} & 0 & 0 & 0 \\
\cline{1-4} \cline{7-10}
1  & 1 & 1 & 1 &  \multicolumn{1}{|c}{0} & 0 &  0 & 0 & \multicolumn{1}{|c}{1} & 1 &  \multicolumn{1}{|c}{0} & 0  \\
1  & 1 & 1 & 1 &  \multicolumn{1}{|c}{0} & 0 &  0 & 0 & \multicolumn{1}{|c}{1} & 1 &  \multicolumn{1}{|c}{0} & 0  \\
\cline{1-4} \cline{9-12}
1  & 1 & 1 & 1 &  \multicolumn{1}{|c}{0} & 0 & 0 & 0 & 0 & 0 & \multicolumn{1}{|c}{1} & 1   \\
1  & 1 & 1 & 1 &  \multicolumn{1}{|c}{0} & 0 &  0 & 0 & 0 & 0 & \multicolumn{1}{|c}{1} & 1 \\
\end{pmatrix}.
\]
For $\alpha=1,2,3,4$ take
\[
\I_\alpha=\J_\alpha = \{1,2,3,4\} \cup \{4+2\alpha-1,\, 4+2\alpha\}
\]
and let $\T=\bigcup_{\alpha=1}^4\I_\alpha \times \J_\alpha$.
Then 
\[
|\T \setminus \Omega| = 3 < 4= \sum_{\alpha=1}^4 (|\I_\alpha|-r)(|\J_\alpha|-r)
\]
and it follows by Theorem \ref{thm:dependent_Asche} that $\Omega$ is dependent. 

To compute $\height(I_\T)$  let $p_\alpha$ be the 6-minor given by  $\I_\alpha \times \J_\alpha$, for $\alpha=1,2,3,4$. 
Any term in $p_\alpha$ is divisible by at least two variables from the set $Z_0=\{z_{ij} \ | \ 1 \le i \le 4, 1\le j \le 4 \} $. 
Consider a DegRevLex term order $\succ$ where the variables $Z_0$ are the smallest among the variables. We claim that is possible to order the remaining variables such that 
\begin{align*}
    p_1= & \ z_{1,6}z_{2,5}z_{5,2}z_{6,1}(z_{3,3}z_{4,4}-z_{3,4}z_{4,3}) \ + \ \text{terms of smaller order} \\
    p_2= & \ z_{3,8}z_{4,7}z_{7,2}z_{8,1}(z_{1,3}z_{2,4}-z_{1,4}z_{2,3}) \ + \ \text{terms of smaller order} \\
    p_3= & \ z_{1,10}z_{2,9}z_{9,4}z_{10,3}(z_{3,1}z_{4,2}-z_{4,1}z_{3,2}) \ + \ \text{terms of smaller order} \\
    p_4= & \ z_{3,12}z_{4,11}z_{11,4}z_{12,3}(z_{1,1}z_{2,2}-z_{1,2}z_{2,1}) \ + \ \text{terms of smaller order}. 
\end{align*}
Indeed, this is obtained by choosing 
\[
z_{1,6},\,z_{2,5},\,z_{5,2},\,z_{6,1},\, z_{3,8},\,z_{4,7},\,z_{7,2},\,z_{8,1},\, z_{1,10},\,z_{2,9},\,z_{9,4},\,z_{10,3},\, z_{3,12},\,z_{4,11},\,z_{11,4},\,z_{12,3}
\]
as the greatest variables. The leading terms of $p_1$, $p_2$, $p_3,$ and $p_4$  w.\,r.\,t.\ $\succ$ are then supported on distinct sets of variables. 
As a consequence, the leading terms of  $p_1, \ldots, p_4$ form a regular sequence of $\k[Z]$. It then follows from Proposition 1.2.12 in \cite{bruns2022determinants} (and the remark that follows it) that $p_1, \ldots,p_4$ is also a regular sequence in $\k[Z]$, whence $\height(I_{\mathcal{T}}) = 4$. 
\end{example}

As all our examples of dependent sets are detected by Theorem \ref{thm:height}, we conjecture that all dependent sets have this property. 

\begin{conjecture}\label{conj:height}
A set $\Omega \subseteq [m] \times [n]$ is a dependent set of the matroid $\mM(r,[m]\times [n])$ if and only if there is a set   $\T \subseteq [m] \times [n]$ such that
\[
\height(I_\T) > |\T \setminus \Omega|. 
\]
\end{conjecture}

We close this section by proving that our methods detect all dependent sets in the cases $r=1$, $\min(m,n)-2 $ and $\min(m,n)-1$. As a consequence
Conjecture \ref{conj:height} is true in these cases.

\begin{proposition} \label{prp:conjecture}
Theorem \ref{thm:height} and Theorem \ref{thm:dependent_Asche} both detect all dependent sets in the matroid $\mM(r,[m]\times [n])$ if $r=1$ or  $r \ge \min(m,n)-2 $.
\end{proposition}
\begin{proof}
We will prove, for all three cases, that if $\Omega$ is dependent then 
\begin{equation}\label{eq:proof_conj_dep_cond}
\text{there exists $\I$ and $\J$ such that $|\Omega \cap ( \I \times \J)| >r(|\I|+|\J|-r)$.}    
\end{equation}
 The claim then follows from Lemma \ref{lemma:size_dep} and Remark \ref{rmk:theorems_detect_basic_dep}. Note that it is enough to prove \eqref{eq:proof_conj_dep_cond} for circuits, i.\,e.\ dependent sets that are minimal with respect to inclusion.

Without loss of generality we assume throughout the proof that $m \le n$. 
In the cases $r=1$ \cite{singer2010uniqueness} and $r=m-1$ \cite{Tsakiris-AMS-2023} the circuits of the matroid $\mM(r, [m] \times [n])$ can be described in terms of bipartite graphs. To a set $\Omega \subseteq [m] \times [n]$, we associate a bipartite graph $G_\Omega$ on vertex set $[m] \sqcup [n]$ and edge set the pairs $\Omega$. The incidence matrix of the graph $G_\Omega$ is the same as the incidence matrix of $\Omega$. The circuits of the matroid $\mM(1, [m] \times [n])$ are the sets $\Omega$ for which the graph $G_\Omega$ is a cycle. In the case $r=m-1$ a set $\Omega$ is a circuit if and only if $G_\Omega$ is a complete bipartite graph $K_{m,m}$.    

{\underline{Assume $r=1$.}}
Let $\Omega$ be a circuit, and take the smallest $\I$ and $\J$ for which $\Omega \subseteq \I \times \J$. The graph $G_\Omega$ is a cycle, so the number of vertices equals the number of edges. The number of vertices in $G_{\Omega}$ is $|\I|+|\J|$, and the number of edges is $|\Omega|$. We see that \eqref{eq:proof_conj_dep_cond} holds as
\[
|\Omega|=|\I|+|\J|>1 \cdot(|\I|+|\J|-1) = r(|\I|+|\J|-r).
\]

{\underline{Assume $r=m-1$.}}
Let $\Omega$ be a circuit. Then $G_\Omega$ is a complete bipartite graph $K_{m,m}$, so we have $\Omega = \I \times \J$ with $|\I|=|\J|=m$. As 
\[
r(|\I|+|\J|-r)=(m-1)(m+1)=m^2-1=|\Omega|-1
\]
we see that \eqref{eq:proof_conj_dep_cond} holds. 

{\underline{Assume $r=m-2$.}}
By Lemma \ref{lemma:indep_reduce} we may remove every column $j$ for which $|\Omega \cap ([m] \times \{j\})|\le r$. Note that the transpose of $\Omega$ 
\[
\Omega^{\top}=\{(j,i) \ | \ (i,j) \in \Omega \} \subseteq [n] \times [m] 
\]
is independent in the matroid $\mM(r,[n] \times [m])$ if and only if $\Omega$ is independent. If we can remove $n-r-1$ columns from $\Omega$, then we may transpose and proceed as in the case $r=m-1$ above. By transposing and applying Lemma \ref{lemma:indep_reduce}, we may remove any row $i$ for which $|\Omega \cap (\{i\} \times [n])|\le r$ and the refer to the case $r=m-1$. After repeated use of Lemma \ref{lemma:indep_reduce} to remove rows and columns we may assume that
\begin{equation}\label{eq:proof_conj_reduced}
|\Omega \cap ([m] \times \{j\})|>r \quad \text{and} \quad |\Omega \cap (\{i\} \times [n])|>r \quad \text{for every $i$ and $j$.}
\end{equation}
After this reduction process we could end up in the situation where $|\Omega|>r(n-m-r)$, in which case \eqref{eq:proof_conj_dep_cond} holds and we are done. 

Claim: Under the condition \ref{eq:proof_conj_reduced}, every $\Omega$ with $|\Omega|=r(n-m-r)$ a base. 

If we can prove this claim we are done, as any $\Omega$ with $|\Omega| < r(n-m-r)$ then can be extended to a base. 

Proof of the claim:
It is proved in \cite{Tsakiris-AMS-2023} that the bases of the matroid $\mM(m-2,[m]\times[n])$ are sets $\Omega$ of size $r(n+m-r)$ that are so called relaxed $(r,r,m)$-SLMF (Definition \ref{definition:SLMF}). 
Since $r=m-2$ we only need to consider $|\I|=m$ and $|\I|=m-1$ in \eqref{eq:rrmSLMF}. For any such $\I$ the assumption $|\Omega \cap ([m] \times \{j\})|>r$ implies $|\Omega \cap (\I \times \{j\})|\ge r$. Hence \eqref{eq:rrmSLMF} simplifies to 
\begin{equation}\label{eq:rrmSLMF'}
\sum_{j=1}^n |\Omega \cap (\I \times \{j\})| \le r(|\I|+n-r) \quad \text{with equality for} \ \I=[m].
\end{equation}
In the case $\I=[m]$ the left hand side of the inequality \eqref{eq:rrmSLMF'} counts the elements of $\Omega$, so \eqref{eq:rrmSLMF} for $\I=[m]$ is equivalent to $|\Omega|=r(m+n-r)$. Let's now consider $\I=[m]\setminus \{i\}$ for some $i$. Then 
\[
\sum_{j=1}^n |\Omega \cap (\I \times \{j\})| = |\Omega \cap (\I \times [n])| = |\Omega|- |\Omega \cap (\{i\} \times [n])|.
\]
Given that $|\Omega|=r(m+n-r)$ the assumption \eqref{eq:proof_conj_reduced} implies 
\[
|\Omega|- |\Omega \cap (\{i\} \times [n])| < r(m+n-r)-r=r(m+n-1-r)
\]
and hence \eqref{eq:rrmSLMF'} holds for all $\I=[m]\setminus \{i\}$. This proves the claim, and the proof is complete. 
\end{proof}

\section{Unique completability} \label{section: Unique Completability}

As already noted in \S \ref{section:Background and notation}, when $\Omega$ is a base of the determinantal matroid, the (scheme-theoretic) generic fiber of the projection morphism $\pi_\Omega$, which takes an $m \times n$ matrix of rank $\le r$ over $\Bbbk$ to the set of its entries indexed by $\Omega$, has dimension zero. This is equivalent to saying that there are finitely many ways to complete those entries in $\k$ (or even in the algebraic closure $\K$ of $\k$) to a matrix of rank $\le r$. In this section we undertake the question of \emph{unique completability}, which concerns the study of those bases (and more generally any $\Omega$), for which $X$ is the only such completion. Throughout this section we will assume hat $\k$ has characteristic zero.

\subsection{Generalities}

We begin by showing that Definitions \ref{dfn:uniquely-completable} and \ref{dfn:generically-uniquely-completable-k} are equivalent, and in fact are equivalent to a purely algebraic statement. Even though Proposition \ref{prp:generically-uniquely-completable-k} might be natural to experts, we have nevertheless taken the liberty of writing a detailed proof. 

\begin{proposition} \label{prp:generically-uniquely-completable-k}
The following are equivalent.
\begin{enumerate}[label=(\roman*)]
\item $\Omega$ is generically uniquely completable at rank $r$ over $\k$.
\item $\k(\bar{Z}_\Omega)=\k(\bar{Z})$.
\item There exists a non-empty open set $U \subseteq \Spec \k[\bar{Z}]$ such that $\pi_\Omega^{-1} \left( \pi_\Omega(X)\right) = \{X\}$ as schemes, for every $X \in U$ ($X$ need not be $\k$-rational or closed). 
\item The morphism $\bar{\pi}_\Omega: \Spec \k[\bar{Z}] \rightarrow \Spec \k[\bar{Z}_\Omega]$ induced by the inclusion $\k[\bar{Z}_\Omega] \subseteq \k[\bar{Z}]$ is birational.
\end{enumerate}
\end{proposition}
\begin{proof}
\emph{(i) $\Rightarrow$ (ii)} We claim that the generic fiber of $\pi_{\Omega,\K}: \Spec \K[\bar{Z}] \rightarrow \Spec \K[\bar{Z}_\Omega]$ is zero-dimensional; here $\pi_{\Omega,\K}$ is induced by $\varphi_\Omega \otimes_\k \K$. The reason is that any fiber of $\pi_{\Omega,\K}$ is the spectrum of a Jacobson ring, and so any prime ideal of it is the intersection of maximal ideals. But by hypothesis there exist fibers which have only one closed point, and so these fibers are necessarily zero-dimensional. As the minimum fiber dimension is attained on an open set, the generic fiber must also be zero-dimensional. 

In particular, the fiber of $\pi_{\Omega,\K}$ over the generic point of $\Spec \K[\bar{Z}]$ is zero-dimensional. As this generic point maps to $\ker(\varphi_{\Omega,\K})$, we have that this fiber is the spectrum of the fiber ring $\K[\bar{Z}] \otimes_{\K[\bar{Z}_\Omega]}\K(\bar{Z}_\Omega)$. Now $\K[\bar{Z}] \otimes_{\K[\bar{Z}_\Omega]} \K(\bar{Z}_\Omega)$ is an Artinian ring, and as it sits in the field $\K(\bar{Z})$ it must be a field itself. This field contains $\K[\bar{Z}]$ thus necessarily $\K[\bar{Z}] \otimes_{\K[\bar{Z}_\Omega]} \K(\bar{Z}_\Omega) = \K(\bar{Z})$. As $\K[\bar{Z}] \otimes_{\K[\bar{Z}_\Omega]} \K(\bar{Z}_\Omega)$ is a finitely generated $\K(\bar{Z}_\Omega)$-algebra coinciding with the field $\K(\bar{Z})$, Zariski's lemma gives that the field extension $\K(\bar{Z}_\Omega) \subseteq \K(\bar{Z})$ is finite. 

For $(i,j) \not\in \Omega$, let $\mu_{ij} \in \K(\bar{Z}_\Omega)[z_{ij}]$ be the minimal polynomial of $\bar{z}_{ij}$ over $\K(\bar{Z}_\Omega)$, and let $d_{ij}$ be its degree. We will show that $d_{ij}=1$. With $\Omega' = \Omega \cup \{i,j\}$, we write $\pi_{\Omega,\K}$ as the composite morphism
$$\pi_{\Omega,\K} : \Spec \K[\bar{Z}] \stackrel{\pi_{\Omega',\K}}{\longrightarrow} \underbrace{\Spec \K[Z_{\Omega'}]}_{\mathbb{A}^{\Omega'}_\K} \stackrel{\pi}{\rightarrow} \underbrace{\Spec \K[Z_{\Omega}]}_{\mathbb{A}^\Omega_\K},$$ 
where $\pi$ corresponds to the inclusion $\K[Z_\Omega] \subseteq \K[Z_{\Omega'}]$; on the level of $\K$-rational points we may think of $\pi$ as dropping the coordinate $(i,j)$. By elimination theory \cite{cox2015ideals}, the closure of the image of $\pi_{\Omega,\K}$ in $\mathbb{A}^\Omega_\K$ is defined by the elimination ideal $\K[Z_\Omega] \cap I_{r+1}(Z)$, and thus this closure is $\Spec \K[\bar{Z}_\Omega]$. Similarly, the closure of the image of $\pi_{\Omega',\K}$ in $\mathbb{A}^{\Omega'}_\K$ is generated by the elimination ideal $\K[Z_{\Omega'}] \cap I_{r+1}(Z)$, and it is $\Spec \K[\bar{Z}_{\Omega'}]$. Consider any lexicographic order on $Z$ that satisfies $z_{\alpha \beta} >z_{ij} > z_{\gamma \delta}$ for every $(\alpha, \beta) \not\in \Omega'$ and $(\gamma, \delta) \in \Omega$. Let us consider the members of a reduced Gr\"obner basis of $I_{r+1}(Z)$ with respect to that order that are supported on $Z_{\Omega'}$; among them call $g_1,\dots,g_s$ the polynomials that involve the variable $z_{ij}$ and $h_1,\dots,h_t$ the remaining ones. Then elimination theory asserts that $\K[Z_{\Omega'}] \cap I_{r+1}(Z) = (g_1,\dots,g_s,h_1,\dots,h_t)$ and $\K[Z_{\Omega}] \cap I_{r+1}(Z)  = (h_1,\dots,h_t)$, and so $$\K[\bar{Z}_\Omega] = \frac{\K[Z_\Omega]}{(h_1,\dots,h_t)}, \, \, \, \K[\bar{Z}_{\Omega'}] = \frac{\K[Z_{\Omega'}]}{(g_1,\dots,g_s,h_1,\dots,h_t)}.$$ 
Note here that we have an inclusion of integral domains
$$\K[\bar{Z}_\Omega] \subseteq \K[\bar{Z}_{\Omega'}] \subseteq \K[\bar{Z}].$$
Factoring $\pi_{\Omega',\K}$ and $\pi_{\Omega,\K}$ through the closures of their images in the respective affine spaces, we get the commutative diagram of morphisms

\begin{center}
\begin{tikzcd}
& \mathbb{A}_\K^{\Omega'} \arrow[r,"\pi"] & \mathbb{A}_\K^{\Omega} \\
\Spec \K[\bar{Z}] \arrow[r]\arrow[ur,"\pi_{\Omega',\K}"] & \underbrace{\Spec \K[\bar{Z}_{\Omega'}]}_{\mathbb{V}(g_1,\dots,g_s,h_1,\dots,h_t)} \arrow[r] \arrow[u] &  \Spec \underbrace{\K[\bar{Z}_{\Omega}]}_{\mathbb{V}(h_1,\dots,h_t)} \arrow[u], \label{eq:composite-morphisms-closures}
\end{tikzcd} 
\end{center}
in which the horizontal arrows are dominant morphisms, while the vertical arrows are closed immersions with the equations that define the corresponding closed subscheme appearing in the notation $\mathbb{V}(\cdot)$. The generic fiber of the composition of the two horizontal arrows is $\K[\bar{Z}] \otimes_{\K[\bar{Z}_\Omega]} \K(\bar{Z}_\Omega)$ and, as we saw above, it has dimension zero; hence 
$$ \dim \K[\bar{Z}_\Omega] = \dim \K[\bar{Z}_{\Omega'}] = \dim \K[\bar{Z}].$$

The kernel of the $\K[\bar{Z}_\Omega]$-algebra epimorphism $$\psi: \K[\bar{Z}_\Omega][z_{ij}] \twoheadrightarrow \K[\bar{Z}_\Omega][\bar{z}_{ij}]=\K[\bar{Z}_{\Omega'}], \, \, \, \, \, \, z_{ij}\mapsto \bar{z}_{ij}$$ 
is the ideal generated by the classes $\tilde{g}_1,\dots,\tilde{g}_s$ in $\K[\bar{Z}_\Omega][z_{ij}]$ of $g_1,\dots,g_s$. Localizing $\psi$ at the multiplicatively closed set $T = \K[\bar{Z}_\Omega] \setminus \{0\}$, we obtain $$\psi_T: \K(\bar{Z}_\Omega)[z_{ij}] \twoheadrightarrow \K(\bar{Z}_\Omega)[\bar{z}_{ij}] = \K(\bar{Z}_{\Omega'}).$$
The kernel of $\psi_{T}$ is the principal ideal generated by the minimal polynomial $\mu_{ij}$ of $\bar{z}_{ij}$ over $\K(\bar{Z}_\Omega)$. Hence in the ring $\K(\bar{Z}_\Omega)[z_{ij}]$ we have $\tilde{g}_\alpha = c_\alpha \mu_{ij}$ for some $c_\alpha \in \K(\bar{Z}_\Omega)[z_{ij}]$, for every $\alpha=1,\dots,s$. We may write each $c_\alpha$ as the quotient of a polynomial in $\K[\bar{Z}_\Omega][z_{ij}]$ by an element $\bar{f}_{\alpha} \in \K[\bar{Z}_\Omega]$. We may further write the $c_\alpha$'s in such a form so that $\bar{f}_{\alpha} = \bar{f}$ for every $\alpha=1,\dots,s$, where $\bar{f} \in \K[\bar{Z}_\Omega]$. Already then, the kernel of the localization  
$$\psi_{\bar{f}}: \K[\bar{Z}_\Omega]_{\bar{f}}[z_{ij}] \twoheadrightarrow \K[\bar{Z}_\Omega]_{\bar{f}}[\bar{z}_{ij}]$$ 
of $\psi$ at $\bar{f}$ is generated by $\mu_{ij}$. The scheme $\Spec \K[\bar{Z}_\Omega][z_{ij}]$ is the affine line $\mathbb{A}^1_{\K[\bar{Z}_\Omega]}$ over the closure of the image of $\pi_{\Omega,\K}$; it is the closed subscheme of $\mathbb{A}^{\Omega'}_\K$ defined by $h_1,\dots,h_t$. What we just established is that inside $\mathbb{A}^1_{\K[\bar{Z}_\Omega]}$ sits the closure of the image of $\pi_{\Omega',\K}$, and it is locally (away from the vanishing locus of $\bar{f}$) a hypersurface defined by $\mu_{ij}$. Incorporating this into the localization of diagram \eqref{eq:composite-morphisms-closures} gives a commutative diagram

\hspace{1in}
\begin{center}
\begin{tikzcd}
    & \mathbb{A}^{\Omega'}_\K \setminus \mathbb{V}(f) \arrow[r] & \mathbb{A}^{\Omega}_\K \setminus \mathbb{V}(f) \\
    & \underbrace{\Spec \K[\bar{Z}_\Omega]_{\bar{f}}[z_{ij}]}_{\mathbb{V}(h_1,\dots,h_t)} \arrow[u] \\
  \Spec \K[\bar{Z}]_{\bar{f}} \arrow[r] &  \underbrace{\Spec \K[\bar{Z}_{\Omega'}]_{\bar{f}}}_{\mathbb{V}(\mu_{ij})} \arrow[u] \arrow[r] & \underbrace{\Spec \K[\bar{Z}_\Omega]_{\bar{f}}}_{\mathbb{V}(h_1,\dots,h_t)} \arrow[uu]
\end{tikzcd}  
\end{center}
in which all horizontal arrows are dominant, and all vertical arrows are closed immersions. 

For $X=(x_{\alpha \beta })$ any closed point of $\Spec \K[\bar{Z}]_{\bar{f}}$, let $x_{ij}'$ be any root of the polynomial $\tilde{\mu}_{ij}$ of $\K[z_{ij}]$ that we obtain by substituting in $\mu_{ij}$ every $z_{\alpha \beta}$ by $x_{\alpha \beta}$ for all $(\alpha,\beta) \in \Omega$; then the closed point of $\mathbb{A}_\K^{\Omega'}$ defined by the data $\{x_{\alpha \beta}: \, (\alpha,\beta) \in \Omega\}$, $x_{ij}'$ is in the closure of the image of $(\pi_{\Omega',\K})_{\bar{f}}$. We proceed as follows in order to further localize so that such points are in fact in the image. By Chevalley's theorem, the image of $(\pi_{\Omega',\K})_{\bar{f}}$ is constructible; as it is dense in its closure $\Spec \K[\bar{Z}_{\Omega'}]_{\bar{f}}$, it must contain a non-empty open set of $\Spec \K[\bar{Z}_{\Omega'}]_{\bar{f}}$. Let $\mathcal{W}'$ be the complement of that open set. As $\Spec \K[\bar{Z}_{\Omega'}]_{\bar{f}}$ is integral, the dimension of $\mathcal{W}'$ is strictly smaller than the dimension of $\Spec \K[\bar{Z}_{\Omega'}]_{\bar{f}}$. Write $\mathcal{W}' = \Spec \K[\bar{Z}_{\Omega'}]_{\bar{f}} / J'$, where $J'$ is the ideal of $\K[\bar{Z}_{\Omega'}]_{\bar{f}}$ that defines $\mathcal{W}'$. Consider the inclusion of integral domains $\K[\bar{Z}_{\Omega}]_{\bar{f}} \subseteq \K[\bar{Z}_{\Omega'}]_{\bar{f}}$; this is an integral ring extension as $\bar{z}_{ij}$ is a root of the monic polynomial $\mu_{ij} \in \K[\bar{Z}_{\Omega}]_{\bar{f}}[z_{ij}]$. Hence with $J$ the contraction of $J'$ to $\K[\bar{Z}_{\Omega}]_{\bar{f}}$, we also have an integral extension of integral domains $\K[\bar{Z}_{\Omega}]_{\bar{f}}/J \subseteq \K[\bar{Z}_{\Omega'}]_{\bar{f}}/J'$. It follows that $\dim \K[\bar{Z}_{\Omega}]_{\bar{f}}/J = \dim \K[\bar{Z}_{\Omega'}]_{\bar{f}}/J'$. Now $\mathcal{W}=\Spec \K[\bar{Z}_{\Omega}]_{\bar{f}}/J$ is the closure of the image of $\mathcal{W}'$ in $\Spec \K[\bar{Z}_\Omega]$ (or equivalently the scheme-theoretic image), and as the dimensions of $\Spec \K[\bar{Z}], \, \Spec \K[\bar{Z}_{\Omega'}], \, \Spec \K[\bar{Z}_{\Omega}]$ agree, we infer that $\mathcal{W}$ is a closed subscheme of $\Spec \K[\bar{Z}_\Omega]$ of strictly smaller dimension. Let $\Spec \K[\bar{Z}_\Omega]_{\bar{g}}$ be a basic non-empty (and thus dense) open subscheme in the complement of $\mathcal{W}$. Its inverse image under $(\pi_{\Omega,\K})_{\bar{f}}$ is the dense open subscheme $\Spec \K[\bar{Z}]_{\bar{f}\bar{g}}$. If now $X=(x_{\alpha \beta})$ is a closed point of $\Spec \K[\bar{Z}]_{\bar{f}\bar{g}}$ and $x_{ij}'$ is as above, then by construction the closed point of $\mathbb{A}_\K^{\Omega'}$ defined by the data $\{x_{\alpha \beta}: \, (\alpha, \beta) \in \Omega\}, \, x_{ij}'$ is in the image of $\pi_{\Omega',\K}$ (for if not, $\{x_{\alpha \beta}: \, (\alpha, \beta) \in \Omega\}$ would be in $\mathcal{W}$).

We have established that for each closed point $X=(x_{\alpha \beta})$ of $\Spec \K[\bar{Z}]_{\bar{f}\bar{g}}$ there are at least as many points in the fiber $\pi_{\Omega,\K}^{-1}(\pi_{\Omega,\K}(X))$ as distinct roots of $\tilde{\mu}_{ij}$ (note here that the fiber consists exclusively of $\K$-rational points because it is zero-dimensional and $\K$ is algebraically closed). As $\mu_{ij}$ is monic of degree $d_{ij}$, so will be its specialization $\tilde{\mu}_{ij}$. We proceed as follows in order to further localize, so that all $d_{ij}$ roots of $\tilde{\mu}_{ij}$ are distinct. Let $\mathfrak{r}$ be the resultant of $\mu_{ij}$ and its derivative with respect to $z_{ij}$. Let $\tilde{\mathfrak{r}}$ be the specialization of $\mathfrak{r}$ by $z_{\alpha \beta} \mapsto x_{\alpha \beta}, \, \forall (\alpha,\beta) \in \Omega$. As $\mu_{ij}$ is irreducible and $\operatorname{char}{\K}=0$, we have that $\mathfrak{r}$ is a non-zero element of $\K[\bar{Z}_\Omega]_{\bar{f}}$. Formation of the resultant commutes with specialization so that the resultant of $\tilde{\mu}_{ij}$ is $\tilde{\mathfrak{r}}$. Moreover, the discriminant of $\tilde{\mu}_{ij}$ is up to sign equal to $\tilde{\mathfrak{r}}$. Recalling that the discriminant of a polynomial in a single variable over a field is zero if and only if the polynomial has roots with higher multiplicity, we see that for every closed point $X=(x_{\alpha \beta})$ of $\Spec (\K[\bar{Z}]_{\bar{f}\bar{g}})_{\mathfrak{r}}$ the polynomial $\tilde{\mu}_{ij} \in \K[z_{ij}]$ has $d_{ij}$ distinct roots and thus the fiber $\pi_{\Omega,\K}^{-1}(\pi_{\Omega,\K}(X))$ contains at least $d_{ij}$ distinct $\K$-rational points. 

We have established that there is a dense open set $U \subseteq \Spec (\K[\bar{Z}])$, such that for any $\K$-rational point $X \in U$, there are at least $d_{ij}$ $\K$-rational points in $\pi_{\Omega,\K}^{-1}(\pi_{\Omega,\K}(X))$. Intersecting $U$ with the dense open set $U'$ defined by the polynomials in Definition \ref{dfn:generically-uniquely-completable-k} (which exists by hypothesis), we get a dense open set $U''$, such that for every $\K$-rational point $X \in U''$ there are at least $d_{ij}$ $\K$-rational points in $\pi_{\Omega,\K}^{-1}(\pi_{\Omega,\K}(X))$ and at the same time $X$ is the only such point. Hence $d_{ij}=1$, and so $\bar{z}_{ij} \in \K(\bar{Z}_\Omega)$. This is true for any $(i,j) \not\in \Omega$ and so $\K(\bar{Z}_\Omega) = \K(\bar{Z})$. To complete the proof we need to show that $\k(\bar{Z}_\Omega) = \k(\bar{Z})$. We first note that $\k(\bar{Z}) \otimes_\k \K$ is a localization of the Noetherian ring $\k[\bar{Z}] \otimes_\k \K = \K[\bar{Z}]$, and thus it is Noetherian. Now the Grothendieck-Sharp formula ((4.2.1.5) at page 349 in \cite{grothendieck1966elements}; see also \cite{sharp1977dimension} more generally) for the Krull dimension of the tensor product of two field extensions gives $$\dim \k(\bar{Z}) \otimes_\k \K = \min\{\operatorname{tr.deg}_\k \k(\bar{Z}), \operatorname{tr.deg}_\k \K\} = 0,$$ 
since $\K$ is the algebraic closure of $\k$. It follows that $\k(\bar{Z}) \otimes_\k \K$ is an Artinian ring. But $\k(\bar{Z}) \otimes_\k \K$ is a subring of $\K(\bar{Z}) = (\k[\bar{Z}] \otimes_\k \K)_{(0)}$ and thus it is an integral domain; it must then be a field. As it contains $\K[\bar{Z}]$, we see that $\k(\bar{Z}) \otimes_\k \K = \K(\bar{Z})$. Similarly, $\k(\bar{Z}_\Omega) \otimes_\k \K = \K(\bar{Z}_\Omega)$. Finally, the $\k(\bar{Z}_\Omega)$-dimension of $\k(\bar{Z})$ coincides with the $\k(\bar{Z}_\Omega)\otimes_\k \K$-dimension of $\k(\bar{Z}) \otimes_\k \K$, the latter being $1$ since $\K(\bar{Z}_\Omega) = \K(\bar{Z})$.

\emph{(ii) $\Rightarrow$ (iii)} Suppose $\k(\bar{Z}_\Omega)=\k(\bar{Z})$. Hence for every $(i,j) \not\in \Omega$ there exists $f_{ij} \in \k(\bar{Z}_\Omega)$ such that $\bar{z}_{ij} = f_{ij}$. Now $\k[\bar{Z}_\Omega]$ is the isomorphic image of $\k[Z_\Omega] / \k[Z_\Omega] \cap I_{r+1}(Z)$ in $\k[\bar{Z}]$ under the map $z_{ij} + I_{r+1}(Z) \cap k[Z_\Omega] \mapsto \bar{z}_{ij}$. This shows that there exist polynomials $p_{ij}, \, q_{ij} \in k[Z_\Omega]$, with $q_{ij} \not\in I_{r+1}(Z)$, such that $$f_{ij} =  \frac{p_{ij} + I_{r+1}(Z)}{q_{ij} + I_{r+1}(Z)} = \frac{\bar{p}_{ij}}{\bar{q}_{ij}}.$$ 
Set $q = \prod_{(i,j) \not\in \Omega} q_{ij}$ and $U = \Spec \k[\bar{Z}]_{\bar{q}}$. As $q \not\in I_{r+1}(Z)$, equivalently $\bar{q} \neq 0$, and $\k[\bar{Z}]$ is an integral domain, we have that $U$ is a non-empty open set of $\Spec \k[\bar{Z}]$. Let $X \in U$ be represented by the prime ideal $\mathfrak{P}$ of $\k[\bar{Z}]$. Let $\mathfrak{p} = \mathfrak{P} \cap \k[Z_\Omega]$ be the contraction of $\mathfrak{P}$ to $\k[Z_\Omega]$ under the map $\varphi_\Omega$; as such $\ker(\varphi_\Omega) \subseteq \mathfrak{p}$. Then $\pi_\Omega^{-1} \left( \pi_\Omega(X)\right) = \Spec \k[\bar{Z}] \otimes_{\k[Z_\Omega]} \kappa(\mathfrak{p})$, where $\kappa(\mathfrak{p}) = \k[Z_\Omega]_\mathfrak{p} / \mathfrak{p}\k[Z_\Omega]_\mathfrak{p} $ is the residue field of $\mathfrak{p}$. We show $\k[\bar{Z}] \otimes_{\k[Z_\Omega]} \kappa(\mathfrak{p}) = \kappa(\mathfrak{P})$. 

Denote by $\k[\bar{Z}]_\mathfrak{p}$ the localization of the $\k[Z_\Omega]$-module $\k[\bar{Z}]$ at the prime ideal $\mathfrak{p}$; this consists of inverting the images in $\k[\bar{Z}]$ of all elements of $\k[Z_\Omega] \setminus \mathfrak{p}$. We have an inclusion of integral domains $\k[\bar{Z}] \subseteq \k[\bar{Z}]_\mathfrak{p} \subseteq \k(\bar{Z}) = \k(\bar{Z}_\Omega)$, where the last equality is true by hypothesis. As $\bar{q} \not\in \mathfrak{P}$ and $\mathfrak{P}$ is prime, we have that $\bar{q}_{ij} \not\in \mathfrak{P}$ for every $(i,j) \not\in \Omega$, whence $q_{ij} \not\in \mathfrak{p}$. Hence already in $\k[\bar{Z}]_\mathfrak{p}$ we have the equality $\bar{z}_{ij} = f_{ij} = \bar{p}_{ij} / \bar{q}_{ij}$, and so $$\k[\bar{Z}]_\mathfrak{p} = \frac{\k[\bar{Z}]_\mathfrak{p}}{(\bar{z}_{ij} - \bar{p}_{ij}/\bar{q}_{ij}: \, (i,j) \not\in \Omega) \, \k[\bar{Z}]_\mathfrak{p}} =\k[\bar{Z}_\Omega]_\mathfrak{p} \cong \frac{\k[Z_\Omega]_\mathfrak{p}}{ \ker (\varphi_\Omega) \,  \k[Z_\Omega]_\mathfrak{p}},$$ 
where we used that $\k[Z_\Omega] / \ker(\varphi_\Omega) \cong \k[\bar{Z}_\Omega]$. It follows that $$\k[\bar{Z}] \otimes_{\k[Z_\Omega]} \kappa(\mathfrak{p}) = 
\frac{\k[\bar{Z}]_\mathfrak{p}}{\mathfrak{p} \k[\bar{Z}]_\mathfrak{p}} \cong \frac{\k[Z_\Omega]_\mathfrak{p}}{(\ker (\varphi_\Omega) + \mathfrak{p})\k[Z_\Omega]_\mathfrak{p}} = \kappa(\mathfrak{p}). $$
As $\pi_\Omega^{-1}(\pi_\Omega(X))$ is the spectrum of a field, it consists of a single point scheme-theoretically, which necessarily must be $\mathfrak{P}$ (note that this proof also shows that $\kappa(\mathfrak{P}) = \kappa(\mathfrak{p}))$.

\emph{(iii) $\Rightarrow$ (i)} Let the open set $U$ of the hypothesis be the complement of the closed subscheme of $\Spec \k[\bar{Z}]$ defined by the ideal $(\bar{p}_1,\dots,\bar{p}_s)$ of $\k[\bar{Z}]$; then the polynomials of Definition \ref{dfn:generically-uniquely-completable-k} can be taken to be $p_1,\dots,p_s \in \k[Z]$. Indeed, let $X=(x_{ij})$ be not a simultaneous root of $p_1,\dots,p_s$ (i.e. it is a $\k$-rational point of $U$) and set $\mathfrak{m}_{X,\Omega}$ to be the maximal ideal $(z_{ij}-x_{ij}: \, (i,j) \in \Omega)$ of $\k[Z_\Omega]$. By hypothesis $\pi_{\Omega}^{-1}(\pi_\Omega(X)) = \{X\}$ as schemes, which is to say that $$\frac{\k[Z]}{I_{r+1}(Z)+\mathfrak{m}_{X,\Omega}k[Z]}=\k[\bar{Z}] \otimes_{\k[Z_\Omega]} \frac{\k[Z_\Omega]}{\mathfrak{m}_{X,\Omega}} = \k.$$ 
Tensoring by $\K$ over $\k$ gives 
$$\frac{\K[Z]}{I_{r+1}(Z)\K[Z]+\mathfrak{m}_{X,\Omega}\K[Z]}=\K[\bar{Z}] \otimes_{\K[Z_\Omega]} \frac{\K[Z_\Omega]}{\mathfrak{m}_{X,\Omega}\K[Z_\Omega]} = \K,$$ 
which is equivalent to saying that $\pi_{\Omega,\K}^{-1}(\pi_{\Omega,\K}(X)) = \{X\}$ as schemes. In particular, $X$ is the only $\K$-rational point in its fiber under $\pi_{\Omega,\K}$.

\emph{(iii)$\Leftrightarrow$ (iv)} This is a consequence of a general fact on birational morphisms of integral schemes of finite type and their fields of functions, e.g. see \href{https://stacks.math.columbia.edu/tag/0552}{Lemma 0552} in \cite{stacks-project}.
\end{proof}

The property of being generically uniquely completable is in fact independent of the ground field (we always assume zero characteristic):

\begin{proposition} \label{prp:generically-uniquely-completable-Q-k}
$\Omega$ is generically uniquely completable at rank $r$ over $\k$ if and only if it is generically uniquely completable at rank $r$ over $\mathbb{Q}$.
\end{proposition}
\begin{proof}
\emph{(if)} Suppose $\Omega$ is generically uniquely completable over $\mathbb{Q}$. Then by Proposition \ref{prp:generically-uniquely-completable-k} we have $\mathbb{Q}(\bar{Z}_\Omega) = \mathbb{Q}(\bar{Z})$. In particular, for any $(i,j) \not\in \Omega$, we have $z_{ij} \in \mathbb{Q}(\bar{Z}_\Omega)$. As $\mathbb{Q}(\bar{Z}_\Omega) \subseteq \k(\bar{Z}_\Omega)$, we readily get $k(\bar{Z}_\Omega) = \k(\bar{Z})$. Again by Proposition \ref{prp:generically-uniquely-completable-k}, $\Omega$ is generically uniquely completable over $k$. 

\emph{(only if)} Suppose $\Omega$ is generically uniquely completable over $\k$. By Proposition \ref{prp:generically-uniquely-completable-k}, the generic fiber of $\pi_{\Omega,\k}$ is finite. Equivalently, $\Omega$ contains a subset $\Omega_0$ which is a base set of the matroid, i.e. $|\Omega_0|
 = r(m+n-r)$ and the map $\k[Z_{\Omega_0}] \rightarrow \k[\bar{Z}]$ is injective. A fortiori the map $\mathbb{Q}[Z_{\Omega_0}] \rightarrow \mathbb{Q}[\bar{Z}]$ is injective. This induces an inclusion $\mathbb{Q}[Z_{\Omega_0}] \cong \mathbb{Q}[\bar{Z}_{\Omega_0}] \subseteq \mathbb{Q}[\bar{Z}]$ of finitely generated $\mathbb{Q}$-domains of the same dimension, rendering the field extension $\mathbb{Q}(\bar{Z}_{\Omega_0}) \subseteq \mathbb{Q}(\bar{Z})$ algebraic. As $\mathbb{Q}(\bar{Z})$ is finitely generated over $\mathbb{Q}(\bar{Z}_{\Omega_0})$, the field extension $\mathbb{Q}(\bar{Z}_{\Omega_0}) \subseteq \mathbb{Q}(\bar{Z})$ is in fact finite, and thus so is $\mathbb{Q}(\bar{Z}_{\Omega}) \subseteq \mathbb{Q}(\bar{Z})$. For $(i,j) \not\in \Omega$, let $\mu_{ij}$ be the minimal polynomial of $\bar{z}_{ij}$ over $\mathbb{Q}(\bar{Z}_{\Omega})$. With $\Omega'=\Omega \cup \{(i,j)\}$, we have that $\mu_{ij}$ generates the kernel of the $\mathbb{Q}(\bar{Z}_{\Omega})$-algebra epimorphism 
$$ \mathbb{Q}(\bar{Z}_{\Omega})[z_{ij}] \twoheadrightarrow \mathbb{Q}(\bar{Z}_{\Omega})[\bar{z}_{ij}] = \mathbb{Q}(\bar{Z}_{\Omega'}),$$
which sends $z_{ij}$ to $\bar{z}_{ij}$. There exists then a non-zero $\bar{f} \in \mathbb{Q}[\bar{Z}_{\Omega}]$ for which we have an exact sequence of $\mathbb{Q}[\bar{Z}_{\Omega}]_{\bar{f}}$-modules

$$0 \rightarrow \mu_{ij}\mathbb{Q}[\bar{Z}_{\Omega}]_{\bar{f}} \rightarrow  \mathbb{Q}[\bar{Z}_{\Omega}]_{\bar{f}}[z_{ij}] \rightarrow \mathbb{Q}[\bar{Z}_{\Omega}]_{\bar{f}}[\bar{z}_{ij}] \rightarrow 0. $$
As $\k$ is faithfully flat over $\mathbb{Q}$, and noting that $\mathbb{Q}[\bar{Z}_{\Omega}]_{\bar{f}} \otimes_{\mathbb{Q}} \k = \k[\bar{Z}_{\Omega}]_{\bar{f}}$, tensoring with $\k$ over $\mathbb{Q}$ gives an  exact sequence of $\k[\bar{Z}_{\Omega}]_{\bar{f}}$-modules
$$0 \rightarrow \mu_{ij}\k[\bar{Z}_{\Omega}]_{\bar{f}} \rightarrow  \k[\bar{Z}_{\Omega}]_{\bar{f}}[z_{ij}] \rightarrow \k[\bar{Z}_{\Omega}]_{\bar{f}}[\bar{z}_{ij}] \rightarrow 0. $$
Further localizing this exact sequence at the multiplicatively closed set $\k[\bar{Z}_{\Omega}]_{\bar{f}} \setminus \{0\}$ finally gives an exact sequence of $\k(\bar{Z}_{\Omega})$-modules 
$$0 \rightarrow (\mu_{ij}) \rightarrow  \k(\bar{Z}_{\Omega})[z_{ij}] \rightarrow \k(\bar{Z}_{\Omega'}) \rightarrow 0. $$
The exactness of this sequence implies that $\mu_{ij}$ is the minimal polynomial of $\bar{z}_{ij}$ as an element of the field $\k(\bar{Z})$ over $\k(\bar{Z}_{\Omega})$. By our hypothesis that $\Omega$ is generically finitely completable over $\k$ and Proposition \ref{prp:generically-uniquely-completable-k}ii), the degree of $\mu_{ij}$ must be equal to one. But then $\bar{z}_{ij} \in \mathbb{Q}(\bar{Z}_\Omega)$. This is true for any $(i,j) \not\in \Omega$ and so $\mathbb{Q}(\bar{Z}_\Omega) = \mathbb{Q}(\bar{Z})$. Again by Proposition \ref{prp:generically-uniquely-completable-k}ii), $\Omega$ is generically uniquely completable over $\mathbb{Q}$. 
\end{proof}

\begin{notation}
We write $\Omega = \bigcup_{j \in[n]} \omega_j \times \{j\}$, where the $\omega_j$'s are subsets of $[m]$. With $\mathcal{J} \subseteq [n]$, we set $\Omega_{\mathcal{J}} = \bigcup_{j \in \mathcal{J}} \omega_j \times \{j\}$ and view $\Omega_{\J}$ as a subset of $[m] \times \J$. The problem of unique completability of $\Omega_\J$ concerns matrices supported at $[m] \times \J$. 
\end{notation}

Observing at least $r$ entries per row or column is necessary for unique completability.

\begin{proposition} \label{prp:<r}
Suppose $\Omega \subseteq [m] \times [n]$ is such that $|\omega_j|<r$ for some $j \in [n]$. Then $\Omega$ is not uniquely completable at rank $r$
\end{proposition}
\begin{proof}
From the proof of Lemma 10 in \cite{Tsakiris-AMS-2023}, $\Omega$ is not even finitely completable, i.e. the matroid rank of $\Omega$ is less than $r(m+n-r)$. 
\end{proof}

On the other hand, rows or columns that are observed at exactly $r$ entries are irrelevant.

\begin{proposition} \label{prp:column-r}
Suppose $\Omega \subseteq [m] \times [n]$ is such that $|\omega_j| = r$ for some $j \in [n]$. Then $\Omega$ is uniquely completable at rank $r$ if and only if $n-1\ge r$ and $\Omega_{[n] \setminus\{j\}}$ is uniquely completable at rank $r$.
\end{proposition}
\begin{proof}
Set $\Omega' = \Omega_{[n] \setminus\{j\}}$ and $Z'$ the matrix that we obtain by deleting the $j$th column of $Z$. Denote by $I'_{r+1}(Z')$ the ideal of $(r+1)$-minors of $Z'$ in the polynomial ring $\k[Z']$. We claim that the generators of the elimination ideal $\k[Z_{\Omega'}] \cap I'_{r+1}(Z')$ also generate the elimination ideal $\k[Z_{\Omega}] \cap I_{r+1}(Z')$. To see this, note that the generators of $\k[Z_{\Omega'}] \cap I'_{r+1}(Z')$ can be computed by running Buchberger's algorithm on the $(r+1)$-minors of $Z'$, under a term order $\succ'$ in $\k[Z']$ that satisfies $z_{\alpha \beta } \succ' z_{\gamma \delta}$ for any $(\alpha,\beta) \not\in \Omega'$ and any $(\gamma,\delta) \in \Omega'$. We can extend $\succ'$ to a term order $\succ$ in $\k[Z]$ which is an elimination order for $\Omega$, and the generators for $\k[Z_{\Omega}] \cap I_{r+1}(Z')$ can be computed by running Buchberger's algorithm on the $(r+1)$-minors of $Z'$ under $\succ$. The point now is that Buchberger's algorithm will behave identically in both cases, because the $(r+1)$-minors of $Z'$ do not involve at all variables from the $j$th column of $Z$. 

By a Gr\"obner basis argument, the proof of Lemma 7 in \cite{Tsakiris-AMS-2023} established that there is an inclusion of sets $\k[Z_\Omega]\cap I_{r+1}(Z) \subset I_{r+1}(Z')$. As the left side always lives in $\k[Z_\Omega]$, we also get $\k[Z_\Omega]\cap I_{r+1}(Z) \subseteq k[Z_\Omega] \cap I_{r+1}(Z')$. As $I_{r+1}(Z') \subset I_{r+1}(Z)$, we conclude that $$\k[Z_\Omega]\cap I_{r+1}(Z) =  \k[Z_\Omega] \cap I_{r+1}(Z').$$
By what we proved in the previous paragraph, $\k[Z_\Omega]\cap I_{r+1}(Z') = (g_1,\dots,g_s)$, where $g_1,\dots,g_s$ are the generators of $\k[Z_{\Omega'}]\cap I'_{r+1}(Z')$; thus also $\k[Z_\Omega]\cap I_{r+1}(Z) = (g_1,\dots,g_s)$. This implies that $\k[\bar{Z}_\Omega] = \k[\bar{Z}_{\Omega'}][z_{\omega_j}]$ is a polynomial ring in the variables $z_{\omega_j}$ over $\k[\bar{Z}_{\Omega'}]$. We thus have $\k(\bar{Z}_\Omega) = \k(\bar{Z}_{\Omega'})(z_{\omega_j})$; in particular the field extension $\k(\bar{Z}_{\Omega'}) \subset \k(\bar{Z}_{\Omega})$ is purely transcendental.

Suppose that $\Omega$ is uniquely completable, i.e. $\k(\bar{Z}_{\Omega}) = \k(\bar{Z})$. A fortiori the generic fiber of $\pi_{\Omega}$ is finite, and by Lemma 7 in \cite{Tsakiris-AMS-2023} so is the generic fiber of $\pi_{\Omega'}$. It follows that the field extension $\k(\bar{Z}_{\Omega'}) \subseteq \k(\bar{Z}')$ is algebraic. Thus for $(\alpha,\beta) \not\in \Omega'$ and $\beta \neq j$ we have that $\bar{z}_{\alpha \beta}$ is algebraic over $\k(\bar{Z}_{\Omega'})$. But $\bar{z}_{\alpha \beta} \in \k(\bar{Z})$, and since $\k(\bar{Z})=\k(\bar{Z}_{\Omega})$ is purely transcendental over $\k(\bar{Z}_{\Omega'})$, we must have that $\bar{z}_{\alpha \beta} \in \k(\bar{Z}_{\Omega'})$; that is $\k(\bar{Z}_{\Omega'}) = \k(\bar{Z}')$ and $\Omega'$ is uniquely completable.

Conversely, suppose $\Omega'$ is uniquely completable. Then for any $(\alpha,\beta) \not\in \Omega'$ and $\beta \neq j$, we have that $\bar{z}_{\alpha \beta} \in \k(\bar{Z}_{\Omega'})$ and a fortiori $\bar{z}_{\alpha \beta} \in \k(\bar{Z}_{\Omega})$. We will be done once we show that $\bar{z}_{ij} \in \k(\bar{Z}_{\Omega})$ for every $i \not\in \omega_j$. Write $\omega_j = \{i_1,\dots,i_r\}$. Let $\J=\{j_1,\dots,j_r\}$ be distinct elements of $[n]$ different from $j$ (such elements exist because of the hypothesis $n-1\ge r$). Denote by $\det(Z_{\omega_j,\J})$ the $r$-minor of $Z$ corresponding to row indices $\omega_j$ and column indices $\J$. By degree considerations, $\det(Z_{\omega_j,\J})$ does not belong to $I_{r+1}(Z)$, and thus $\det(\bar{Z}_{\omega_j,\J})$ is a non-zero element of $\k[\bar{Z}']$, hence a non-zero element of $\k(\bar{Z}') = \k(\bar{Z}_{\Omega'})$. For $i \not\in \omega_j$, it is trivial that $\det(\bar{Z}_{\omega_j \cup \{i\}, \J \cup \{j\}})=0$ in $\k(\bar{Z})$. This is a relation of the form $$\det(\bar{Z}_{\omega_j,\J}) \bar{z}_{ij} + c_1 \bar{z}_{i_1j}+\cdots+c_r \bar{z}_{i_r j} = 0,$$
where $c_1,\dots,c_r$ are elements of $\k(\bar{Z}')=\k(\bar{Z}_{\Omega'})$. We conclude that $\bar{z}_{ij} \in \k(\bar{Z}_\Omega)$. 
\end{proof}

Unique completability ``descends" from higher to lower ranks.

\begin{proposition} \label{prp:rank-descent-full}
If a set $\Omega$ is uniquely completable at rank $r$ then $\Omega$ is uniquely completable at rank $r' < r$.
\end{proposition}
\begin{proof}
The hypothesis and Definition \ref{dfn:generically-uniquely-completable-k} imply that $\k(\bar{Z}_\Omega) = \k(\bar{Z})$, where the $\bar{\cdot}$ notation indicates modulo $I_{r+1}(Z)$. As in the beginning of the proof of Proposition \ref{prp:generically-uniquely-completable-k}, this implies that for every $(i,j) \not\in \Omega$, there exist $p_{ij}, \, q_{ij} \in k[Z_\Omega]$ such that $q_{ij} z_{ij} - p_{ij} \in I_{r+1}(Z)$. Now, for any $r' < r$ we have the inclusion of determinantal ideals $I_{r+1}(Z) \subset I_{r'+1}(Z)$. Hence $q_{ij} z_{ij} - p_{ij} \in I_{r'+1}(Z)$ and thus $\k(\bar{Z}_\Omega) = \k(\bar{Z})$, where now the $\bar{\cdot}$ notation indicates modulo $I_{r'+1}(Z)$. 
\end{proof}

On the other hand, that unique completability does not ``ascend" in rank, can be seen by dimension reasons.

\begin{example}
Proposition 4 in \cite{Tsakiris-AMS-2023} describes bases of $\mathcal{M}(r,[m] \times [n])$ which are uniquely completable at rank $r$. These are not even finitely completable at rank $r+1$, because their size is equal to 
\[\rank \mathcal{M}(r,[m] \times [n])  = r(m+n-r)< (r+1)(m+n-r-1) = \rank \mathcal{M}(r+1,[m] \times [n]).\]
\end{example}


In good circumstances ``stitching" is possible if and only if ``excision" is possible.

\begin{proposition} \label{prp:stitching-excision-unique}
Identifying $\Omega$ with its incidence matrix, $\Omega$ is uniquely completable at rank $r$ if and only if $\Omega'$ is uniquely completable at rank $r+1$, where
$$ \Omega' = 
\begin{bmatrix}
\Omega & 1_{m \times 1} \\
1_{1 \times n} & 1
\end{bmatrix}.$$ 
\end{proposition}
\begin{proof}
Let $Z'$ be an $(m+1) \times (n+1)$ matrix of indeterminates, defined as 
$$Z' = \begin{bmatrix}
z_{11} & \cdots & z_{1n} & z_{1(n+1)}' \\
\vdots & \cdots & \vdots & \vdots \\
z_{m1} & \cdots & z_{mn} & z_{m(n+1)}' \\
z_{(m+1)1}' &  \cdots & z_{(m+1)n}' & z_{(m+1)(n+1)}'
\end{bmatrix}$$
that is, $Z'$ is obtained by appending to $Z$ an extra row and column of indeterminates. Denote by $I_{r+2}'(Z')$ the ideal of $(r+2)$-minors of $Z'$ in the polynomial ring $\k[Z']$. Then it is a classical tool in the theory of determinantal rings (e.g., see Proposition 2.4 in \cite{BrunsVetter:1988}), that in the localized ring $\k[Z']_{z_{(m+1)(n+1)}'}$, where the variable $z_{(m+1)(n+1)}'$ has been inverted, the ideal $I_{r+2}'(Z')$ extends to the ideal $I'_{r+1}(Y)$ of $(r+1)$-minors of the $m \times n$ matrix $$Y = (y_{ij})_{i \le m, j \le n}  = (z_{ij} - z'_{(m+1)j}z_{(m+1)(n+1)}'^{-1} z'_{i(n+1)})_{i \le m, j \le n};$$
that is $I_{r+2}'(Z') \k[Z']_{z_{(m+1)(n+1)}'} = I'_{r+1}(Y)$. The elements $$\{z_{ij} - z'_{(m+1)j}z_{(m+1)(n+1)}'^{-1} z'_{i(n+1)}\}_{i \le m, j \le n} \cup \{z_{ij}'\}$$ 
are jointly algebraically independent over $k$; this can seen by noting that these are $(m+1)(n+1)$ elements which generate the subalgebra $\k[Z']_{z_{(m+1)(n+1)}'}$ of $\k(Z')$, whose field of fractions is $\k(Z')$, and thus of transcendence degree $(m+1)(n+1)$ over $\k$. In particular, the entries of $Y$ are algebraically independent over $\k$. We will additionally denote by $I_{r+1}(Y)$ the ideal of $(r+1)$-minors of $Y$ in the algebra $\k[Y]$.

\emph{(only if)} Suppose that $\Omega$ is uniquely completable at rank $r$. With $A=(a_{ij})$ an $m \times n$ matrix of indeterminates, we have that $\k(\bar{A}_\Omega) = \k(\bar{A})$, where now $\k(\bar{A})$ is the fraction field of the intergral domain $\k[\bar{A}] = \k[A] / I_{r+1}(A)$. Thus for every $(\alpha,\beta) \in [m] \times [n] \setminus \Omega$ we have a relation of the form 
\begin{equation}\label{eq:relation}
q_{\alpha \beta}(A_\Omega) a_{\alpha \beta} + p_{\alpha \beta}(A_\Omega)  = \sum_{\I,\J} f_{\alpha \beta, \I, \J}(A) \det(A_{\I,\J}),
\end{equation}
where $q_{\alpha \beta}(A_\Omega), p_{\alpha \beta}(A_\Omega) \in \k[A_\Omega]$ with $q_{\alpha \beta} \not\in I_{r+1}(A)$, $f_{\alpha \beta, \I, \J}(A) \in \k[A]$, and the summation ranges over all subsets $\I \subseteq [m], \, \J \subseteq [n]$ of cardinality $r+1$. 

Consider the $\k$-algebra homomorphism $\k[A] \rightarrow \k[Y]$ that sends $a_{ij}$ to $y_{ij}$. As the $y_{ij}$'s are algebraically independent over $\k$, this establishes an isomorphism $\k[A] \cong \k[Y]$ and carries \eqref{eq:relation} to the relation 
$$ q_{\alpha \beta}(Y_\Omega) y_{\alpha \beta} + p_{\alpha \beta}(Y_\Omega)  = \sum_{\I,\J} f_{\alpha \beta, \I, \J}(Y) \det(Y_{\I,\J})$$
in $\k[Z']_{z_{(m+1)(n+1)}'}.$
As each $\det(Y_{\I,\J})$ is now an element of $I_{r+2}'(Z') \k[Z']_{z_{(m+1)(n+1)}'}$, each $f_{\alpha \beta, \I, \J}(Y)$ is an element of $\k[Z']_{z_{(m+1)(n+1)}'}$, $q_{\alpha \beta}(Y_\Omega) = \tilde{q}_{\alpha \beta}(Z'_{\Omega'})$ and $p_{\alpha \beta}(Y_\Omega) = \tilde{p}_{\alpha \beta}(Z'_{\Omega'})$ for suitable $\tilde{q}_{\alpha \beta}, \tilde{p}_{\alpha \beta} \in \k[Z'_{\Omega'}]_{z_{(m+1)(n+1)}'}$, we have 

$$\tilde{q}_{\alpha \beta}(Z'_{\Omega'}) z_{\alpha \beta} + \underbrace{\tilde{p}_{\alpha\beta}(Z'_{\Omega'}) -\tilde{q}_{\alpha \beta}(Z'_{\Omega'})z'_{(m+1)\beta}z_{(m+1)(n+1)}'^{-1} z'_{\alpha(n+1)}}_{\in \k[Z'_{\Omega'}]_{z_{(m+1)(n+1)}'}} \in I_{r+2}'(Z') \k[Z']_{z_{(m+1)(n+1)}'}.$$
As $\tilde{q}_{\alpha \beta}(Z'_{\Omega'}) \not\in I_{r+2}'(Z') \k[Z']_{z_{(m+1)(n+1)}'}$, this shows that $\bar{z}_{\alpha \beta} \in \k(\bar{Z}'_{\Omega'})$, where now $\bar{\cdot}$ indicates modulo $I_{r+2}'(Z')$; i.e. $\Omega'$ is uniquely completable at rank $r+1$.

\emph{(if)} Suppose $\Omega'$ is uniquely completable at rank $r+1$; then 
for every $(\alpha,\beta) \not\in \Omega'$ there are polynomials $q_{\alpha \beta}, p_{\alpha \beta} \in k[Z'_{\Omega'}]$ with $q_{\alpha \beta} \not\in I_{r+2}(Z')$, and  $f'_{\alpha \beta, \I', \J'} \in \k[Z']$ such that 
$$ q_{\alpha \beta} z_{\alpha \beta} + p_{\alpha \beta}  = \sum_{\I',\J'} f'_{\alpha \beta, \I', \J'} \det(Z'_{\I',\J'}),$$ 
where now the summation extends over all subsets $\I' \subseteq [m+1]$ and $\J' \subseteq [n+1]$ of cardinality $r+2$. 

Let us view this relation as a relation in $R=\k[Z']_{z_{(m+1)(n+1)}'} = \k[Y,\{z_{ij}'\}]_{z_{(m+1)(n+1)}'}$. Every $\det(Z'_{\I',\J'})$ is an element of $I_{r+1}(Y)$ and thus we can write it as a combination of $(r+1)$-minors of $Y$ with coefficients in $R$. Hence for every $\I \subseteq [m]$ and $\J\subseteq [n]$ of cardinality $r+1$ there are $f_{\alpha \beta, \I, \J} \in R$ such that 

$$\sum_{\I',\J'} f'_{\alpha \beta, \I', \J'} \det(Z'_{\I',\J'}) = \sum_{\I,\J} f_{\alpha \beta, \I, \J} \det(Y_{\I,\J}).$$
We also write $q_{\alpha \beta}, p_{\alpha \beta}$ as polynomials in the $z_{ij}'$'s with coefficients in $\k[Y]$ (allowing $z_{(m+1)(n+1)}'$ to appear with a negative exponent). Then not all coefficients of $q_{\alpha \beta}$ are in $I_{r+1}(Y)$, otherwise $q_{\alpha \beta}$ would be in $I_{r+2}(Z')$. So let $u$ be a monomial of $q_{\alpha \beta}$, whose coefficient, call it $q_{\alpha \beta u}$, is not in $I_{r+1}(Y)$. Denote by $\tilde{p}_{\alpha \beta u}$ the coefficient of $u$ in 
$$p_{\alpha \beta} + q_{\alpha \beta} z'_{(m+1)\beta}z_{(m+1)(n+1)}'^{-1} z'_{\alpha(n+1)},$$
and by $c_u$ the coefficient of $u$ in $\sum_{\I,\J} f_{\alpha \beta, \I, \J} \det(Y_{\I,\J})$. As the elements $Y,\{z_{ij}'\}$ are algebraically independent over $\k$, we must have $q_{\alpha \beta u} y_{\alpha \beta} + \tilde{p}_{\alpha \beta u}  = c_u$. Note that $c_u \in I_{r+1}(Y)$, $q_{\alpha \beta u} \not\in I_{r+1}(Y)$, and $q_{\alpha \beta u}, \, \tilde{p}_{\alpha \beta u} \in \k[Y_\Omega]$; this implies $\bar{y}_{\alpha \beta} \in \k(\bar{Y}_\Omega)$. 
\end{proof}

A similar statement can be made for base sets.

\begin{proposition} \label{prp:stitching-excision-bases}
Identifying $\Omega$ with its incidence matrix, $\Omega$ is a base of the matroid $\mathcal{M}(r,[m]\times [n])$ if and only if $\Omega'$ is a base of the matroid $\mathcal{M}(r+1,[m+1]\times [n+1])$, where
$$ \Omega' = 
\begin{bmatrix}
\Omega & 1_{m \times 1} \\
1_{1 \times n} & 1
\end{bmatrix}.$$ 
\end{proposition}
\begin{proof}
We adopt the notation in the proof of Proposition \ref{prp:stitching-excision-unique}. First of all, $\Omega$ and $\Omega'$ have the right number of ones to qualify for base sets of the respective matroids, since 
\begin{align*}
\rank \mathcal{M}(r,[m]\times [n]) +m+n+1 &= r(m+n-r)+m+n+1 \\
&= (r+1)m+(r+1)n-r^2+1\\
&=(r+1)(m+1)+(r+1)(n+1)-2(r+1)-r^2+1 \\
& = (r+1)(m+1)+(r+1)(n+1)-(r+1)^2\\
& = (r+1)(m+1+n+1-r-1) \\
&= \rank \mathcal{M}(r+1,[m+1]\times [n+1]).
\end{align*} Secondly, tensoring the map 
$\alpha: \k[Y_{\Omega}] \rightarrow \k[Y] / I_{r+1}(Y)$ over $\k$ with $\k[\{z_{ij}'\}]_{z'_{(m+1)(n+1)}}$, and using the fact that $I'_{r+1}(Y) = I_{r+2}(Z')\k[Z']_{z'_{(m+1)(n+1)}}$, gives the localization $$\beta: \k[Z'_{\Omega'}]_{z'_{(m+1)(n+1)}} \rightarrow \k[Z']_{z'_{(m+1)(n+1)}}/I'_{r+1}(Y) =\left(\k[Z']/I_{r+2}(Z')\right)_{z'_{(m+1)(n+1)}} $$
of the map $\k[Z'_{\Omega'}] \rightarrow \k[Z']/ I_{r+2}(Z')$. As $\k[\{z_{ij}'\}]_{z'_{(m+1)(n+1)}}$ is faithfully flat over $\k$, we have that $\alpha$ is injective if and only if $\beta$ is injective. Moreover, $\beta$ is injective if and only if $\k[Z'_{\Omega'}] \rightarrow \k[Z'] / I_{r+2}(Z')$ is injective.
\end{proof}

We have the following obvious criterion for unique completability.

\begin{proposition}\label{prop:our-game}
Suppose there is an ordering of $[m] \times [n]$ such that $(i,j)>(\alpha,\beta)$ for every $(i,j) \not\in \Omega$ and every $(\alpha,\beta) \in \Omega$, and such that for every $(i,j) \not\in \Omega$, there is an $(r+1)$-minor that is supported on $(i,j)$ together with elements strictly less than $(i,j)$. Then $\Omega$ is uniquely completable at rank $r$.
\end{proposition}

Unfortunately, not all uniquely completable bases come in this nice form. 

\begin{example}\label{ex:uniq-comp}
    The set 
    \[
    \Omega = \begin{bmatrix}
        0 & 1 & 1 & 1 & 0 & 0 \\
        1 & 0 & 1 & 1 & 0 & 0 \\
        1 & 1 & 0 & 0 & 1 & 1 \\
        1 & 1 & 0 & 0 & 1 & 1 \\
        0 & 0 & 1 & 1 & 0 & 1 \\
        0 & 0 & 1 & 1 & 1 & 0 \\
    \end{bmatrix}
    \] 
    is a base of the matroid $\mM(2,[6] \times [6])$, which is uniquely completable.
    This can be seen by computing the ``circuit polynomials'' for all $\T=\Omega \cup \{(i,j)\}$ with $(i,j) \notin \Omega$ using Macaulay2 \cite{M2}. That is, we compute the irreducible polynomial generating the principal ideal $E_\T$. These polynomials all turn out linear in the variable $z_{ij}$. Hence, there is precisely one solution for each $z_{ij}$ and we have unique completability. 

Proposition \ref{prop:our-game} does not apply to this example as there is no $3$-minor containing only one entry outside $\Omega$.
\end{example}

\subsection{Dimension and codimension $1$}

In ``dimension 1" we have a complete answer to the unique completability question.

\begin{theorem}
The following are equivalent.
\begin{enumerate}[label=(\roman*)]
\item $\Omega$ is uniquely completable at rank $1$.
\item $\Omega$ is finitely completable at rank $1$.
\item $|\omega_j| \ge 1 \ \forall \, j \in [n]$, and $\Omega$ supports a relaxed $(1,1,m)$-SLMF.
\end{enumerate}
\end{theorem}
\begin{proof}
\emph{(i) $\Rightarrow$ (ii)} Trivial.

\emph{(ii) $\Rightarrow$ (iii)} Suppose $\Omega$ is finitely completable at rank $1$. Then $\Omega$ contains a base of the algebraic matroid $\mathcal{M}(1,[m] \times [n])$. By Theorem 1 in \cite{Tsakiris-AMS-2023}, $\Omega$ supports a relaxed $(1,1,m)$-SLMF. 

\emph{(iii) $\Rightarrow$ (i)} Suppose that $\Omega$ supports a relaxed $(1,1,m)$-SLMF $\Omega_0$. In particular, there exists a subset $\J \subseteq [n]$ of column indices, such that $\Omega_0 \subseteq \Omega_{\J}$. By Proposition 4 in \cite{Tsakiris-AMS-2023}, $\Omega_0 \subseteq [m] \times \J$ is uniquely completable, and so $\k(\bar{Z}_{\Omega_0}) = \k(\bar{Z}_{[m] \times \J})$. Now let $(\alpha,\beta) \not\in \Omega$ with $\beta \not\in \J$. As by hypothesis $|\omega_{\beta}| \ge 1$, there exists $\alpha' \in [m]$ such that $(\alpha',\beta) \in \Omega$. Pick any $j \in \J$; in $\k(\bar{Z})$ we trivially have the relation $\bar{z}_{\alpha' j} \bar{z}_{\alpha \beta}  - \bar{z}_{\alpha' \beta} \bar{z}_{\alpha j}=0$ or equivalently $\bar{z}_{\alpha \beta} = (\bar{z}_{\alpha' \beta} \bar{z}_{\alpha j}) / \bar{z}_{\alpha' j}$. As $\bar{z}_{\alpha' \beta} \in \bar{Z}_{\Omega}$ and both $\bar{z}_{\alpha j}, \, \bar{z}_{\alpha' j} \in \k(\bar{Z}_{[m] \times \J}) = \k(\bar{Z}_{\Omega_0}) \subseteq \k(\bar{Z}_{\Omega})$, we infer $\bar{z}_{\alpha \beta} \in k(\bar{Z}_{\Omega})$. 
\end{proof}

For the other extreme of ``codimension $1$", we also have a complete answer.

\begin{theorem} 
With the convention that $m \le n$, the following are equivalent.
\begin{enumerate}[label=(\roman*)]
\item $\Omega$ is uniquely completable at rank $m-1$.
\item $\Omega$ is finitely completable at rank $m-1$.
\item $|\omega_j| \ge m-1 \ \forall \, j \in [n]$, and there exists $\J \subseteq [n]$ with $|\J|=m-1$ and $[m] \times \J \subseteq \Omega$.
\end{enumerate}
\end{theorem}
\begin{proof}
\emph{(i) $\Rightarrow$ (iii)} Suppose $\Omega$ is uniquely completable at rank $m-1$. By Proposition \ref{prp:<r} we have $|\omega_j| \ge m-1, \forall j \in [n]$. Let $\J' \subseteq [n]$ be the column indices for which $|\omega_j| > m-1$ (thus necessarily $|\omega_j|=m$). By inductive application of Proposition \ref{prp:column-r} we have that $|\J'| \ge m-1$ and $\Omega_{\J'}$ is uniquely completable at rank $m-1$. By definition of $\J'$ we have $\Omega_{\J'}=[m] \times \J'$; hence there exists a subset $\J \subseteq \J'$ of cardinality $m-1$ such that $\Omega_{\J} = [m] \times \J$.

\emph{(iii) $\Rightarrow$ (i)} Suppose that $|\omega_j| \ge m-1, \forall j \in [n]$ and there exists $\J \subseteq [n]$ with $|\J| = m-1$ and $\Omega_\J = [m] \times \J$. Pick any $(\alpha,\beta) \not\in \Omega$ (necessarily $\beta \not\in \J$). Trivially $\det(\bar{Z}_{[m], \J \cup \{\beta\}}) = 0$ in $\k(\bar{Z})$, and from this relation one immediately gets that $\bar{z}_{\alpha \beta} \in \k(\bar{Z}_\Omega)$.

\emph{(i) $\Rightarrow$ (ii)} Trivial. 

\emph{(ii) $\Rightarrow$ (i)} Suppose $\Omega$ is finitely completable; then it contains a base $\Omega'  \subseteq \Omega$ of the algebraic matroid $\mathcal{M}(m-1,[m] \times [n])$. Let $\J \subseteq [n]$ be indexing the $\omega'_{j}$'s for which $|\omega'_{j}| = m$. As $\Omega'$ is finitely completable, necessarily $|\omega'_{j}| \ge m-1, \, \forall j \in [n]$ (Lemma 10 in \cite{Tsakiris-AMS-2023}); now Lemma 7 in \cite{Tsakiris-AMS-2023} gives that $\Omega'_{\J} = \Omega' \cap ([m] \times \J)$ is a base of the algebraic matroid $\mathcal{M}(m-1,[m] \times \J)$, and thus must have cardinality equal to the rank of that matroid. On the other hand, by definition of $\J$ we have $\Omega'_{\J} = [m] \times \J$. We obtain the equality $(m-1)(m+(|\J|)-m+1) = m (|\J|)$, which is equivalent to $|\J| = m-1$. 
\end{proof}

\begin{remark}
The equivalence of (i) and (ii) in the previous two theorems also follows from Theorem 1.8 in \cite{Larson2024B}, or more simply, as pointed out by Matt Larson, from the existence of square-free universal Gr\"obner bases for $2$-minors and maximal minors.
\end{remark}

\subsection{Codimension $2$}

We now turn to ``codimension 2".

\begin{lemma} \label{lem:codim2-4x4}
All base sets of the matroid $\mM(2,[4]\times [4])$ are uniquely completable at rank $2$, except those that are row and column permutations of
$$
\begin{bmatrix}
0 & 1 & 1 & 1 \\
1 & 0 & 1 & 1 \\
1 & 1 & 0 & 1 \\
1 & 1 & 1 & 0
\end{bmatrix}.$$ 
\end{lemma}
\begin{proof}
We recall that a characterization of base sets in codimension $2$ is known from  Proposition 2 in \cite{Tsakiris-AMS-2023}. Every base $\Omega$ of $\mathbb{M}(2,[4]\times [4])$ needs to have (in terms of incidence matrices) $12$ ones. Moreover, an $\Omega$ with $12$ ones is a base if and only if every row and every column has at most $2$ zeros.

We consider first the case where some row or column has two zeros. Without loss of generality we may assume that the last column of $\Omega$ has exactly $2$ zeros (and thus exactly $2$ ones). By Proposition \ref{prp:column-r}, we have that $\Omega$ is uniquely completable at rank $2$ if and only if $\Omega'$ is uniquely completable at rank $2$, where $\Omega'$ is obtained by deleting from $\Omega$ the last column. Now $\Omega'$ is a base of the matroid $\mathbb{M}(2,[4]\times [3])$ and thus has exactly $2$ zeros. Since $\Omega'$ is a $4 \times 3$ matrix with exactly $2$ zeros, there must exist a row (say the last row) and a column (say the last column) full of ones. Calling $\Omega''$ the matrix that we obtain by deleting this row and column from $\Omega'$, Proposition \ref{prp:stitching-excision-unique} gives that $\Omega'$ is uniquely completable at rank $2$ if and only if $\Omega''$ is uniquely completable at rank $1$. By Proposition \ref{prp:stitching-excision-bases}, $\Omega''$ is a base of the matroid $\mathbb{M}(1,[3]\times [2])$ and thus its two zeros do not occur in the same row. As there is a row full of ones, every zero occurs in a $2$-minor where the rest elements are ones, and it follows that $\Omega''$ is uniquely completable at rank $1$; equivalently $\Omega$ is uniquely completable at rank $2$.

We now consider the case where there is a single zero per row and column in $\Omega$, and thus we may assume that 
$$\Omega = 
\begin{bmatrix}
0 & 1 & 1 & 1 \\
1 & 0 & 1 & 1 \\
1 & 1 & 0 & 1 \\
1 & 1 & 1 & 0
\end{bmatrix}.$$ 
We will show that this $\Omega$ is not uniquely completable. For this, we will produce an irreducible polynomial which lies in $I_3(Z)$, and it is quadratic in $z_{11}$ with coefficients in $k[Z_\Omega]$. Necessarily, this must be the circuit polynomial of $\bar{Z}_{\Omega \cup \{(1,1)\}}$. Towards this end, the $3$-minor $[124|123]$ of $Z$ involves only the variables $z_{11}$ and $z_{22}$ outside of $\Omega$ and it is quadratic and square-free in these variables. The same is true for the $3$-minor $[123|124]$ of $Z$. Eliminating the quadratic term $z_{11} z_{22}$ gives a degree-$1$ polynomial in the variables $z_{11}$ and $z_{22}$, which we can then use to eliminate the variable $z_{22}$ from say $[124|123]$. This calculation yields that $z_{43} p \in I_3(Z)$, where 
\begin{align*}
p&=z_{11}^2 (z_{23}z_{42}z_{34}-z_{24} z_{32} z_{43}) + z_{11}p_1+p_2 \quad \text{for} \\
p_1&=-z_{21}z_{13}z_{42}z_{34}-z_{41}z_{12}z_{23}z_{34}+z_{21}z_{14}z_{32}z_{43} +z_{31}z_{12}z_{24}z_{43}-z_{23}z_{42}z_{31}z_{14}+z_{41}z_{13}z_{24}z_{32} \\
p_2&=-z_{41}z_{13}z_{21}z_{14}z_{32}-z_{41}z_{13}z_{31}z_{12}z_{24}-z_{21}z_{12}z_{43}z_{31}z_{14}+z_{21}z_{12}z_{41}z_{13}z_{34}+z_{21}z_{13}z_{42}z_{31}z_{14}\\
&+z_{41}z_{12}z_{23}z_{31}z_{14}.
\end{align*}
 As $I_3(Z)$ is a prime ideal, $p \in I_3(Z)$. We now show that $p$ is irreducible in $k[Z_{\Omega \cup \{(1,1)\}}]$. If $p$ were not irreducible, it would be the product of two degree-$1$ polynomials in $z_{11}$ with coefficients in $k[Z_\Omega]$, and thus for any specialization of the variables $Z_\Omega$ in $\mathbb{Q}$ which gives a quadric, the specialized quadric would have its roots in $\mathbb{Q}$. 
However, for the specialization 
\begin{align*}
z_{12}, \, z_{41} \mapsto 0 \\
z_{13}, \, z_{14}, \, z_{21}, \,  z_{24}, \, z_{31}, \, z_{32}, \, z_{42}, z_{43} \mapsto 1 \\
z_{23}, \, z_{34} \mapsto 2
\end{align*} we obtain the quadric $3 z_{11}^2 - 3 z_{11} +1$, which has negative discriminant.
\end{proof}

By what we have proved so far, the $m \times m$ codimension-2 readily follows.

\begin{theorem}\label{prp:codim-2-square}
All bases of the matroid $\mM(m-2,[m]\times [m])$ are uniquely completable except those for which no two zeros occur in the same row or column.
\end{theorem}
\begin{proof}
First of all we note that any base $\Omega$ has exactly four zeros. Thus there exist $m-4$ rows and columns full of ones. Let $\Omega'$ be the $4 \times 4$ matrix that we obtain by deleting these rows and columns from $\Omega$. Applying $m-4$ times Proposition \ref{prp:stitching-excision-bases} yields that $\Omega$ is uniquely completable at rank $m-2$ if and only if $\Omega'$ is uniquely completable at rank $2$. Now Lemma \ref{lem:codim2-4x4} asserts that $\Omega'$ is uniquely completable at rank $2$ unless no two zeros occur in the same row or column. 
\end{proof}

A generalization of Theorem \ref{prp:codim-2-square} to arbitrary matrix size $m \times n$ is not available at the moment, since it seems to require a case by case treatment, as in Lemma \ref{lem:codim2-4x4}, of patterns that depend on $n$.

\begin{example}\label{ex:not-uniq-comp}
The set $\Omega$ given by the incidence matrix
$$ \begin{bmatrix}
1 & 1 & 1 & 1 & 0 & 0 \\
1 & 1 & 1 & 0 & 1 & 1 \\
1 & 1 & 0 & 1 & 1 & 1 \\
1 & 0 & 1 & 1 & 1 & 1 \\
0 & 1 & 1 & 1 & 1 & 1
\end{bmatrix}$$
 is a base of the matroid $\mM(3,[5] \times [6])$ which is not uniquely completable.
A computation in Macaulay2 shows that the elimination ideal $E_\Omega \cup \{(5,1)\}$ is generated by one polynomial $p$ which has degree 3 in the variable $z_{51}$. The polynomial $p$ is supported in all variables $Z_\Omega \cup \{z_{51}\}$, which makes it a circuit polynomial, provided it's irreducible. If $p$ is not irreducible it would have a factor which is linear in $z_{51}$, and any specialization to $\mathbb{Q}$ would provide $p$ with a rational solution for $z_{51}$. However, the substitutions 
\begin{align*}
z_{11}, \, z_{12}, \, z_{22}, \, z_{26}, \, z_{31}, \, z_{34}, \, z_{35}, \, z_{43}, \, z_{45}, \, z_{53}, \, z_{56} & \mapsto 0 \\
z_{13}, \, z_{14}, \, z_{21}, \,  z_{23}, \, z_{25}, \, z_{32}, \, z_{36}, \, z_{41}, z_{44}, \, z_{46}, \, z_{52}, \, z_{55} & \mapsto 1 \\
z_{54} & \mapsto -1
\end{align*} 
gives $p=z_{51}^3-2z_{51}+2$ which has no rational roots. 
\end{example}

\section{Diagonal and anti-diagonal bases} \label{section: Diagonal Bases}
This final section is devoted to a special class of bases of the determinantal matroid, which we may call \emph{diagonal} or \emph{anti-diagonal} bases. These follow naturally from the well-known algebraic-combinatorial structure of determinantal ideals. 

 On $[m] \times [n]$ we consider a partial order in which $(i,j) \le (\alpha, \beta)$ if and only if $i \le \alpha$ and $j \ge \beta$. Let $\xi_1 \le \xi_2$ be points in $[m] \times [n]$. An \emph{anti-diagonal path} from $\xi_1$ to $\xi_2$ is a maximal chain with respect to this partial order whose minimum is $\xi_1$ and maximum is $\xi_2$. A \emph{diagonal path} is defined analogously by instead using the partial order defined as  $(i,j) \le (\alpha,\beta)$ if $i \le \alpha$ and $j \le \beta$. 
 We say that a family of (anti-)diagonal paths are non-intersecting if the paths are pairwise disjoint.  

\begin{theorem} \label{thm:anti-diagonal}
Any family of non-intersecting anti-diagonal paths from $(1,n)$, $(2,n)$, $\dots$, $(r,n)$ to $(m,1), \, (m,2), \dots, (m,r)$ is a base of the matroid $\mM(r,[m]\times [n])$.
\end{theorem}
\begin{proof}
Consider the lexicographic term order $\succ$ on $k[Z]$ induced by 
$$z_{11} \succ z_{12} \succ \dots \succ z_{1n} \succ z_{21} \succ z_{22} \succ \dots \succ z_{2n} \succ z_{31} \succ \dots \succ z_{m1} \succ z_{m2} \succ \dots \succ z_{mn},$$ already used in the proof of Lemma \ref{lemma:small_block}. This is known as a diagonal term order, as it has the property that the leading term of any minor of $Z$ is the term associated to the product of the variables along the main diagonal of the minor. By Sturmfels \cite{sturmfels1990grobner} (see also Theorem 4.3.2 in \cite{bruns2022determinants}) the $(r+1)$-minors form a Gr\"obner basis of $I_{r+1}(Z)$ for this diagonal order. 

The independence complex $\Delta(\operatorname{in}_{\succ}(I_{r+1}(Z)))$ of $\operatorname{in}_{\succ}(I_{r+1}(Z))$ is by definition the set of $Z_\Omega$'s that are algebraically independent in the quotient ring $k[Z] / \operatorname{in}_{\succ}(I_{r+1}(Z))$. As $\operatorname{in}_{\succ}(I_{r+1}(Z))$ is a square-free monomial ideal, we have that $Z_\Omega \in \Delta(\operatorname{in}_{\succ}(I_{r+1}(Z)))$ if and only if $\prod_{(i,j) \in \Omega} z_{ij} \not\in \operatorname{in}_{\succ}(I_{r+1}(Z))$. In particular, $\Delta(\operatorname{in}_{\succ}(I_{r+1}(Z)))$ is the simplicial complex associated to the Stanley-Reisner ring $k[Z] / \operatorname{in}_{\succ}(I_{r+1}(Z))$. This simplicial complex is well understood: it is pure of dimension $\dim k[Z] / \operatorname{in}_{\succ}(I_{r+1}(Z))-1=r(m+n-r)-1$ and its facets are precisely the families of non-intersecting anti-diagonal paths of the statement; see for instance Proposition 4.4.1 in \cite{bruns2022determinants}. We conclude that every $\Omega$ given by non-intersecting anti-diagonal paths as in the statement is algebraically independent modulo $\operatorname{in}_{\succ}(I_{r+1}(Z))$ and of size $r(m+n-r)$. 

By Kalkbrener \& Sturmfels \cite{kalkbrener1995initial} the independence complex of $I_{r+1}(Z)$ is the union of the independence complexes of all initial ideals of $I_{r+1}(Z)$ across all term orders. It follows that $\Delta(\operatorname{in}_{\succ}(I_{r+1}(Z)))$ is a subcomplex of $\Delta(I_{r+1}(Z))$, and so the said $\Omega$ remains algebraically independent modulo $I_{r+1}(Z)$, and in fact it is a base. 
\end{proof}

\begin{example}
The set $\Omega$ in Example \ref{ex:not-1rm-SLMF} with $n=m=9$ and $r=4$ consists of four non-intersecting anti-diagonal paths, and is a base by Theorem \ref{thm:anti-diagonal}. The figure below illustrates the four paths. 

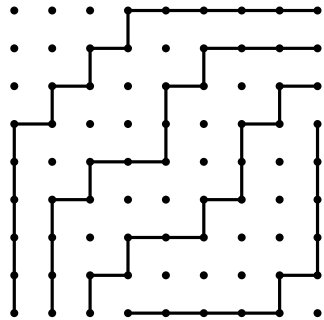
\begin{figure}[h]
\begin{tikzpicture}[scale=0.5]
 \foreach \i in {0,...,8}
    \foreach \j in {0,...,8}
        \fill (\i,\j) circle (3pt);

\draw[very thick] (0,0) -- (0,5);
\draw[very thick] (0,5)--(1,5)--(1,6)--(2,6)--(2,7)--(3,7)--(3,8)--(8,8);
\draw[very thick] (1,0)--(1,3)--(2,3)--(2,4)--(4,4)--(4,6)--(5,6)--(5,7)--(8,7);
\draw[very thick] (2,0)--(2,1)--(3,1)--(3,2)--(5,2)--(5,3)--(6,3)--(6,5)--(7,5)--(7,6)--(8,6);
\draw[very thick] (3,0)--(7,0)--(7,1)--(8,1)--(8,5);
\end{tikzpicture}
\end{figure}

\end{example}

\begin{remark} \label{prp:diagonal}
Applying the permutation $j \mapsto n-j+1$ to the columns, the ``anti-diagonal bases'' in Theorem \ref{thm:anti-diagonal} become ``diagonal bases''. Thus if $\Omega$ is a family of non-intersecting diagonal paths from $(1,1)$, $(2,1)$, $\dots$, $(r,1)$ to $(m,n-r+1), \, (m,n-r+2), \dots, (m,n)$, then $\Omega$ is a base of the matroid $\mM(r,[m]\times [n])$.
\end{remark}

\begin{question}[Aldo Conca] \label{que:Aldo}
What is the degree of the projection map $\pi_\Omega$ for each the bases of Theorem \ref{thm:anti-diagonal}? In particular, which of these bases are uniquely completable?
\end{question}

Using any of the two partial orders on $[m] \times [n]$ introduced above, we recall that a subset $\Omega \subseteq [m] \times [n]$ is a \emph{ladder} if $P, P' \in \Omega$ and $P' \le P'' \le P$ implies $P'' \in \Omega$. The sets $\Omega_1$ and $\Omega_2$ in Example \ref{ex:anti-diagonal-bases} below are examples of ladders. With this, we give a partial answer to Question \ref{que:Aldo}.

\begin{proposition} \label{prp:diagonal-ladders}
Let $\Omega$ be a family of non-intersecting anti-diagonal paths from $(1,n)$, $(2,n)$, $\dots$, $(r,n)$ to $(m,1), \, (m,2), \dots, (m,r)$. Suppose that $\Omega$ is a ladder. Then $\Omega$ is a uniquely completable base of the matroid $\mM(r,[m]\times [n])$.
\end{proposition}  
\begin{proof}
The claim follows by applying Proposition \ref{prop:our-game}, after endowing $[m] \times [n]$ with any total order that extends the following partial order. First we declare every element of $\Omega$ to be less than any element outside $\Omega$. Then we totally order the elements below the ladder from smallest to largest, going from left to right and from top to bottom; i.e. if $(\alpha,\beta)$ is the first lower inside corner of the ladder (e.g. see Fig. 3 in \cite{conca1995ladder} for the definition), then $(\alpha+1,\beta+1)$ is the smallest element below the ladder, $(\alpha+1,\beta+2)$ is the second smallest element and so on. Similarly, we totally order the elements above the ladder from smallest to largest, going from right to left and from bottom to top. 
\end{proof}

\begin{example} \label{ex:anti-diagonal-bases}
Let $m=4, \, n=4$ and $r=2$, and consider the four sets
\begin{align*}
\Omega_1 = \begin{pmatrix}
\underline{1} & \underline{1} & \underline{1} & \underline{1} \\
\underline{1} & 1 & 1 & 1 \\
\underline{1} & 1 & 0 & 0 \\ 
\underline{1} & 1 & 0 & 0 \\
\end{pmatrix}, \,
\Omega_2 = \begin{pmatrix}
0 & \underline{1} & \underline{1} & \underline{1} \\
\underline{1} & \underline{1} & 1 & 1 \\
\underline{1} & 1 & 1 & 0 \\ 
\underline{1} & 1 & 0 & 0 \\
\end{pmatrix}, \,
\Omega_3 = \begin{pmatrix}
0 & \underline{1} & \underline{1} & \underline{1} \\
0 & \underline{1} & 0 & 1 \\
\underline{1} & \underline{1} & 0 & 1 \\ 
\underline{1} & 1 & 1 & 1 \\
\end{pmatrix}, \,
\Omega_4 = \begin{pmatrix}
0 & \underline{1} & \underline{1} & \underline{1} \\
\underline{1} & \underline{1} & 0 & 1 \\
\underline{1} & 0 & 1 & 1 \\
\underline{1} & 1 & 1 & 0
\end{pmatrix}
\end{align*} Each $\Omega_i$ consists of $r=2$ anti-diagonal paths; here we have underlined the $1$'s of one of the paths to distinguish it from the other. As seen, the two paths do not intersect, and thus, according to Theorem \ref{thm:anti-diagonal}, they are bases of the matroid $\mathcal{M}(2,[4] \times [4])$. It follows from Lemma \ref{lem:codim2-4x4} that $\Omega_1, \, \Omega_2, \, \Omega_3$ are uniquely completable, while $\Omega_4$ is not. 
That $\Omega_1$ and $\Omega_2$ are uniquely completable also follows from Proposition \ref{prp:diagonal-ladders}, as they are ladders.
The set $\Omega_3$ is not a ladder under any permutation of the rows and columns. 
This shows that not all uniquely completable anti-diagonal bases are of the form of Proposition \ref{prp:diagonal-ladders}.
\end{example}

It is an intriguing issue to understand the rest of the uniquely completable bases in this family, as well as the interplay between the degree of the projection map $\pi_\Omega$ for the bases that are not uniquely completable, and their finer combinatorial structure. 

\bibliographystyle{amsalpha}
\bibliography{Nicklasson-Tsakiris-2025}

\end{document}